\documentclass[11pt,twoside, letter]{article}
\usepackage{amsthm, amsmath, amssymb, amsfonts}
\usepackage{graphicx, epsfig}

\usepackage{bbm}
\usepackage{mathrsfs}

\usepackage{epstopdf}

\usepackage{enumerate}

\usepackage[colorlinks=true, pdfstartview=FitV, linkcolor=blue,
            citecolor=blue, urlcolor=blue]{hyperref}
\usepackage[usenames]{color}

\usepackage{url}
\makeatletter\def\url@leostyle{%
 \@ifundefined{selectfont}{\def\UrlFont{\sf}}{\def\UrlFont{\scriptsize\ttfamily}}} \makeatother

\usepackage{mathtools}
\mathtoolsset{showonlyrefs=true}

\usepackage[margin=1.0in, letterpaper]{geometry}

\newtheorem{theorem}{Theorem}
\newtheorem{proposition}[theorem]{Proposition}
\newtheorem{lemma}[theorem]{Lemma}
\newtheorem{corollary}[theorem]{Corollary}

\theoremstyle{definition}

\newtheorem{example}[theorem]{Example}
\theoremstyle{remark}
\newtheorem{remark}[theorem]{Remark}

\numberwithin{equation}{section}
\numberwithin{theorem}{section}

\definecolor{Red}{rgb}{0.9,0,0.0}
\definecolor{Blue}{rgb}{0,0.0,1.0}

\def\cA{\mathcal{A}}

\def\cD{\mathcal{D}}

\def\cF{\mathcal{F}}

\def\cH{\mathcal{H}}

\def\cN{\mathcal{N}}

\def\bE{\mathbb{E}}

\def\bN{\mathbb{N}}

\def\bP{\mathbb{P}}

\def\bR{\mathbb{R}}

\def\sF{\mathscr{F}}

\def\mV{\mathsf{V}}

\newcommand{\wt}{\widetilde}
\newcommand{\wh}{\widehat}

\newcommand{\set}[1]{\{#1\}}            
\renewcommand{\mid}{\;|\;}              
\newcommand{\Mid}{\;\Big | \;}          
\newcommand{\norm}[1]{ \| #1 \| }       

\DeclareMathOperator{\dif}{d \!}        

\DeclareMathOperator{\Var}{Var}          

\title{A note on parameter estimation for discretely sampled SPDEs}

\author{
    Igor Cialenco \\[-0.3ex]
    \url{cialenco@iit.edu}  \\[-0.9ex]
    \url{http://math.iit.edu/\~igor}
 \and
    Yicong Huang \\[-0.3ex]
    \url{yhuang37@hawk.iit.edu}  
 \and \\[-1.9ex]
        {\footnotesize Department of Applied Mathematics, Illinois Institute of Technology} \\
        {\footnotesize 10 W 32nd Str, REC Room 220, Chicago, IL 60616, USA}\\
        }

\date{ {\small  
First Circulated: October 4, 2017 \\
This Version: May 30, 2019 \\[0.5em] 
Forthcoming in Stochastics and Dynamics
}}

\begin{document}

\maketitle

\vspace{-2em}

\smallskip

{\footnotesize
\begin{tabular}{l@{} p{350pt}}
  \hline \\[-.2em]
  \textsc{Abstract}:    \ & We consider a parameter estimation problem  for one dimensional stochastic heat equations, when data is sampled discretely in time or spatial component.  We prove that, the real valued parameter next to the Laplacian (the drift), and the constant parameter in front of the noise (the volatility) can be consistently estimated under somewhat surprisingly minimal information. Namely, it is enough to  observe the solution at a fixed time and on a discrete spatial grid, or at a fixed space point and at discrete time instances of a finite interval, assuming that the mesh-size goes to zero. The proposed estimators have the same form and asymptotic properties regardless of the nature of the domain - bounded domain or whole space. The derivation of the estimators and the proofs of their asymptotic properties are based on computations of power variations of some relevant stochastic processes.  We use elements of Malliavin calculus to establish the asymptotic normality properties in the case of bounded domain. We also discuss the joint estimation problem of the drift and volatility coefficient. We conclude with some numerical experiments that illustrate the obtained theoretical results.
          \\[0.5em]
\textsc{Keywords:}      \ &  p-variation, power variation, statistics for SPDEs, discrete sampling, stochastic heat equation, inverse problems for SPDEs, Malliavin calculus.   \\
\textsc{MSC2010:}       \ & 60H15, 35Q30, 65L09 \\[1em]
  \hline
\end{tabular}
}

\section{Introduction}
Consider the following (parabolic) Stochastic Partial Differential Equations (SPDEs)
\begin{equation}\label{eq:Abstract}
\dif u(t) =( \theta \mathcal{A}_1 + \cA_0)u(t) \dif t + \sigma (\mathcal{M} u(t) + g(t))\dif W(t) \, ,
\end{equation}
where $\cA_0, \mathcal{A}_1, \mathcal{M}$ are some (linear or nonlinear) operators acting in suitable Hilbert spaces,
$g$ is an adapted vector-valued function, $W$ is a  cylindrical Brownian motion, and $\theta$ and $\sigma$ are unknown parameters (to be estimated) belonging to a subset of real line. Implicitly we will assume that \eqref{eq:Abstract} is parabolic and admits a unique solution, although usually this has to be established on a case by case basis.

Major part of the existing literature on statistical inference for SPDEs (estimating $\theta$ and $\sigma$) lies within the spectral approach, where it is assumed that one path of the first $N$ Fourier modes of the solution is observed continuously  over a finite interval of time.
In this case, the coefficient $\sigma$ can be determined explicitly and exactly, similar to the case of finite dimensional diffusions, by employing quadratic variation type arguments, and due to the fact that a path is observed continuously in time. A general method of estimating $\theta$ is to construct Maximum Likelihood Estimators (MLEs) based on the information revealed by the first $N$ Fourier modes, and prove that these estimators satisfy the desired statistical properties, such as consistency, asymptotic normality, and efficiency, as $N$ increases. For MLE based estimators applied to nonlinear SPDEs see for instance \cite{IgorNathanAditiveNS2010}. For other type of estimators, assuming the same observation scheme, see \cite{CialencoGongHuang2016}.
We refer the reader to the monograph \cite[Chapter~6]{LototskyRozovsky2017Book} and recent survey paper \cite{Cialenco2018} for a comprehensive overview of literature on statistical inference for SPDEs. Beyond spectral approach, the literature on parameter estimation for SPDEs is limited, and only few papers are devoted to discretely sampled SPDEs \cite{PiterbargRozovskii1997,Markussen2003,PospivsilTribe2007}. The main goal of this note is to contribute to these efforts and study the parameter estimation problem for parabolic SPDEs, when data is sampled discretely in physical domain. It has to be mentioned that by the time this manuscript was moving through the review process, several works appeared that study a similar problem, albeit by quite different methodologies. Simultaneously and independently of the present work, in \cite{BibingerTrabs2017,BibingerTrabs2019} the authors consider a second order linear parabolic SPDE on a bounded domain and driven by an additive noise, and study the problem of estimating the volatility coefficient, or integrated volatility in the semi-parametric setup, assuming that the solution is sampled on a discrete time and space grid. Using mixing theory approach, the authors prove consistency and asymptotic normality of the proposed estimators when the time and/or space mesh size goes to zero. On the other hand, in \cite{Chong2019}, the author studies  the analogues problem for similar equations but on whole space, by using methods rooted in the statistical inference for semi-martingale, and proving the asymptotic properties of the estimators when the time mesh vanishes.
One way to deal with discretely sampled data, is to discretize or approximate the MLEs using the available discrete data, and show that the statistical properties are preserved. This approach is addressed in \cite{CialencoDelgado-VencesKim2019}, where the authors study the drift estimation problem when the Fourier coefficients are observed at discrete time points. On the other hand, if we assume that the solution itself is observed at some space-time grid points, one needs to approximate additionally the Fourier modes. To best of our knowledge, a rigourous asymptotic analysis of this idea is still to be done.

In this paper we consider the stochastic heat equation, in dimension one, driven by an additive space-time noise, and assume that the solution $u$ is observed at some discrete space-time points. We do not rely on spectral approach, but rather derive some suitable representations of the solution to obtain the corresponding estimators. The major focus of the paper is to find consistent and asymptotically normal estimators for $\theta$ and/or $\sigma$ by using minimal amount of information. The main findings can be summarised as follows:
\begin{itemize}
  \item The drift $\theta$ or volatility $\sigma$ can be estimated assuming that the solution is observed just at one  (interior) space point and at discrete time points of a finite time interval, with the time mesh-size going to zero. Similarly, to estimate $\theta$ or $\sigma$ it is enough
      to observe the solution at one time instant and discretely on a spacial grid of a finite interval, with mesh diameter going to zero.
  \item For both sampling schemes the estimators are consistent and asymptotically normal, yielding a rate of convergence $1/\sqrt{n}$, where $n$ is the number of points in the grid.
  \item Both, the bounded domain or the whole space are considered. Due to the local nature of the estimators, they remain  the same regardless of the shape of the the domain, and exhibit the same asymptotic properties.
  \item We derive consistent joint estimators for $\theta$ and $\sigma$.
  \item New useful representation of the solution for the case of bounded domain are obtained.
\end{itemize}
The key idea of the proposed method is based  on an intuitively clear observation: the $p$-variation (or the power variation) of a stochastic process is invariant with respect to smooth perturbations. Hence, if the $p$-variation of a process $X$ can be computed by an explicit formula, and the parameter of interest enters non-trivially into this formula, one can derive consistent estimators of this parameter. However, since the $p$-variation of the perturbed process $X+Y$ remains the same, given that $Y$ is smooth enough, then the same estimator remains consistent assuming that $X+Y$ is observed. Analogous arguments remain valid for asymptotic normality property. The formal result is presented in Section~\ref{sec:setup}.
Thus, it remains to find suitable representations of the solution $u$ as a sum of two processes, which itself is an interesting problem. As already mentioned, we focus our study on two sampling schemes. In Section~\ref{sec:FixedX} we study the sampling scheme with fixed one space point and sampling discretely in time, for both bounded and unbounded domain. Section~\ref{sec:fixedTime} is dedicated to observations at one time instance and discrete space sampling. The case of the whole space is easiest to deal with, thanks to ready available representations of the solution; see \cite[Section~3]{Khoshnevisan2014BookSBMS} for details. It turns out that  for any fixed instance of time $t>0$, the solution as a function of $x\in\bR$ can be represented as a scaled two-sided Brownian motion plus a smooth process. Similarly, if we fix a spacial point, then the solution is a smoothly perturbed scaled fractional Brownian motion. Similar estimators where studied in \cite{PospivsilTribe2007} where the authors considered the heat equation on $\bR$ driven by a  multiplicative noise, and prove consistency by  different methods from ours.  The case of bounded domain is more intricate, due to lack of results on the representations of the solution. In Proposition~\ref{th:bounedFixedX} we prove that the solution can be represented as a sum of a smooth process and a zero-mean Gaussian process with known finite fourth variation. In contrast to the existing works, we use elements of Malliavin  calculus, as well as a version of the central limit theorem from \cite{NualartOrtiz2008}, to establish a central limit type theorem for the fourth variation of the solution. Consequently, we derive weakly consistent estimators for $\theta$ and $\sigma$, and prove their asymptotic normality.
Similar methodology of using  Malliavin  calculus technics to establish central limit theorem can be found in \cite{Corcuera2012}, although applied to similar processes but with  a simpler covariance structure. The case of bounded domain and fixed time is dealt by using Karhunen--Lo\`eve type expansions.

The importance of the chosen two sampling schemes is twofold. First, note that using existing methods based on spectral approach, to estimate consistently the drift $\theta$ the solution has to be observed (discretely or continuously) on entire domain and over a finite interval of time. In contrast, the results obtained here guarantee consistent estimation of both drift and volatility under significantly further information, revealing an important property of the statistical experiment, which essentially is exploiting the singularity of the probability measures generated by the solution for different values of the parameters. Secondly, in many practical applications the solution indeed is observed only at some a priori specified space points and at high time-frequency; e.g. temperature of a heated body, velocity of a turbulent flow, instantaneous forward rates where the space variable corresponds to time until maturity. On the other hand, to incorporate the additional information of observing the solution at several space points and discretely in time, or more generally by observing the solution at a discrete space-time grid, it is enough to take the (weighted) average of the proposed estimators; see Section~\ref{sec:space-time}. Finally, using a combination of the two sampling scheme (one fixed space point, and one fixed time point), we develop novel joint estimators that allow to find simultaneously $\theta$ and $\sigma$; see Section~\ref{sec:space-time}. Consistency of such estimators follows from the main results, while the asymptotic normality remains an open problem.

We conclude the paper with several numerical examples that validate the obtained theoretical results; see Section~\ref{sec:numerical}.
To streamline the presentation, some of the proofs and auxiliary technical results are moved to Appendix~\ref{appendix:auxiliaryResults}.

\section{Setup of the problem and preliminary results}\label{sec:setup}

Let $(\Omega,\sF,\set{\sF_t}_{t\geq 0},\bP)$ be a stochastic basis satisfying the usual assumptions, and let $G$ be either a bounded domain in $\bR$, say $G=[0,\pi]$ or the whole real line $G=\bR$. We consider the following stochastic partial differential equation on $H = L^2(G)$
\begin{equation}\label{eq:mainSPDE}
\begin{cases}
\dif u(t,x) =\theta u_{xx}(t,x)\dif t+\sigma \dif W(t,x),\quad x\in G,\quad t>0, \\
 u(0,x) = 0,
\end{cases}
\end{equation}
where  $\theta,\sigma$ are some positive constants, and $W(t,x)$ is a space-time white noise, namely a zero mean Gaussian field with covariance structure $\bE[W(t,x)W(s,y)]=\min(x,y)\min(t,s)$ for any $x,y\in G, \ t,s\geq0$. For the case of bounded domain, $G=[0,\pi]$, we also assume zero boundary conditions $u(t,0)=u(t,\pi)=0, \ t>0$. It is well known that the solution to \eqref{eq:mainSPDE} exists and is unique \cite{ChowBook,LototskyRozovsky2017Book}.

As usual, everywhere below, all equalities and inequalities between random variables, unless otherwise noted, will be  understood in the $\bP$-a.s. sense. The notations $\xrightarrow[]{\cD}$ will be used for convergence in distribution, while $\xrightarrow[]{\bP}$ or $\bP\!-\!\lim$ will stand for convergence in probability.

We assume that $\theta\in\Theta\subset (0,+\infty)$ and $\sigma\in\mathbf{S}\subset (0,+\infty)$ are the (unknown) parameters of interest.
The main focus of this work are the following sampling schemes\footnote{For simplicity of writing, we assume that the sampling points form a uniform grid. Generally speaking most of the results hold true assuming only that the mesh-size of the grid goes to zero, and with some of the `almost sure convergence' replaced with `convergence in probability'.}:
\begin{enumerate}[(A)]
  \item \textit{Fixed space and discrete time.}  For a fixed  $x$ from the interior of $G$, and given time interval $[c,d]\subset(0,+\infty)$, the solution $u$ is observed at points $\{(t_{i},x), \, i=1,\ldots,n\}$, where $t_{i}:=c+ (d-c)i/n,\ i=0,1,\ldots,n$.
  \item \textit{Fixed time and discrete space.} For a fixed instant of time $t>0$, and given interval $[a,b]\subset G$, the solution $u$ is observed at points  $(t,x_j), \ j=1,\ldots,m$, with $x_{j}=a+(b-a)j/m,\quad j=0,1,\ldots,m$.

\end{enumerate}
The main goal of this paper is to derive consistent estimators for the parameters $\theta$ and $\sigma$ under these sampling schemes, and to study the asymptotic properties of these estimators. In addition to these statistical experiments, we also investigate the estimation of $\theta$ and $\sigma$ when the solution is sampled at space-time grid points. Moreover, using the specific structure of the original estimators under sampling scheme (A) and (B) we are able to derive joint estimators for $\theta$ and $\sigma$ by using the measurements of the solution once by sampling scheme (A) and once by sampling scheme (B).

In what follows, we will use the notation  $\Upsilon^m(a,b)=\set{a_j \mid a_j= a+(b-a)j/m,\ j=0,1,\ldots,m}$ for the uniform partition of size $m$ of a given  interval $[a,b]\subset\bR$. For a given stochastic process $X$ on some interval $[a,b]$, and $p\geq 1$, we will denote by $\mV^p_m(X; [a,b])$ the sum
$$
\mV^p_m(X; [a,b]) := \sum_{j=1}^{m} |X(t_j) - X(t_{j-1})|^p,
$$
where $t_j\in\Upsilon^m(a,b)$. Correspondingly,
\begin{align*}
\mV^p(X; [a,b]) &:=  \lim_{m\to\infty}\mV^p_m(X; [a,b]), \quad \bP-\textrm{a.s.}, \\
\mV^p_\bP(X; [a,b]) &:=  \bP\!-\!\lim_{m\to\infty}\mV^p_m(X; [a,b]),
\end{align*}
will denote the $p$-variation of $X$ on $[a,b]$, in $\bP$-a.s. sense and respectively in probability. If no confusions arise, we will simply write $\mV^p(X)$, and $\mV^p_m(X)$ instead of
$\mV^p(X; [a,b])$ and $\mV^p_m(X; [a,b])$; same applies to $\mV^p_\bP(X)$.

The next result shows that the $p$-variation is invariant with respect to smooth perturbations.
\begin{proposition}\label{prop:QVarSmoothPert}
  Let $X(t),Y(t), \, t\in[a,b]$, be stochastic processes with continuous paths, and assume that the process $Y$ has $C^1[a,b]$ sample paths, and
 there exists $p>1$, such that $0<\mV^p(X)<\infty$.  Then,
\begin{equation}\label{eq:SamePVarSmoothPert}
  \mV^p (X+Y; [a,b]) = \mV^p(X; [a,b]).
\end{equation}
Similarly, if $0<\mV^p_\bP(X)<\infty$, then
\begin{equation}\label{eq:SamePVarSmoothPertbP}
	\mV_\bP^p (X+Y; [a,b]) = \mV_\bP^p(X; [a,b]).
\end{equation}
 If in addition, there exist $\alpha, \sigma_0>0$ such that, $\alpha+1/p <1$, 
 \begin{equation}\label{eq:CLTSmoothPert}
   n^\alpha \left(  \mV^p_n(X; [a,b])  - \mV^p(X; [a,b]) \right) \xrightarrow[n\to\infty]{\cD}\cN(0,\sigma^2_0),
 \end{equation}
 then
 \begin{equation}\label{eq:CLTSmoothPert2}
   n^\alpha \left(  \mV^p_n(X+Y; [a,b])   - \mV^p(X; [a,b]) \right) \xrightarrow[n\to\infty]{\cD}\cN(0,\sigma^2_0).
 \end{equation}
Moreover, if $Y$ has $C^2[a,b]$ sample paths, and \eqref{eq:CLTSmoothPert} holds for $p=2$ and $\alpha=1/2$, then \eqref{eq:CLTSmoothPert2} holds true too, with $p=2, \alpha=1/2$.
\end{proposition}
\noindent The proof is deferred to Appendix~\ref{appendix:auxiliaryResults}.

This result allows to construct directly consistent and asymptotically normal estimators for some parameter entering the true law of the perturbed process $X+Y$, given that the $p$-variation $\mV^p(X;[a,b])$ of the unperturbed process $X$ depends non-trivially on the parameter of interest, and this dependence can be computed explicitly.

\begin{remark}\label{rem:1}
As we will see later, finding such suitable representations of the solution $u$ of \eqref{eq:mainSPDE} will be at the core of this study. For some cases such representations are ready available, while for other cases these representations have to be established, which is one of the major task of this work.
\end{remark}

\begin{example}\label{ex:BM}
Let $B$ be a two-sided Brownian motion, and $Y$ be a process with a $C^2(\bR)$ version, and consider the stochastic process
$$
Z(x)=\sqrt{\beta} B(x)+Y(x),\quad x\in\bR,
$$
where $\beta$ is a positive, unknown parameter.

Assume that $Z$ is observed at grid points $\Upsilon^m(a,b)$, for some interval $[a,b]\subset \bR$.
In view of \eqref{eq:SamePVarSmoothPert}, $\mV^2(Z;[a,b]) =  \mV^2(\sqrt{\beta}B;[a,b]) = \beta(b-a)$.
Consequently, the estimator
\begin{align}
\wh{\beta}_{m}=\frac1{b-a}\sum_{j=1}^{m}\left(Z(x_{j})-Z(x_{j-1}) \right)^{2},
\end{align}
is a consistent  estimator of $\beta$, namely  $\lim_{m\rightarrow\infty}\wh{\beta}_{m}=\beta$, $\bP$-a.s..
Moreover, it is well known (cf.~\cite{Nourdin2008,AaziziEs-Sebaiy2016}) that
$$
\sqrt{m}( \mV_m^2(B,[a,b]) - (b-a)) \xrightarrow[m\to\infty]{\cD}\cN(0,2 (b-a)^2),
$$
and thus, by Proposition~\ref{prop:QVarSmoothPert},  the estimator $\widehat\beta_m$ is asymptotically normal, with the rate of convergence given by
\begin{align}
\sqrt{m}(\wh{\beta}_{m}-\beta)\xrightarrow[m\to\infty]{\cD}\cN(0,2\beta^{2}). \label{eq:CLTforQuadVariationBS}
\end{align}
\end{example}

\begin{example}\label{ex:fBM}
Let $B^{H}$ be a fractional Brownian Motion (fBM) with Hurst index $H=\frac14$, and $Y$ be a process with continuously differentiable paths in $(0,+\infty)$. Assume that $\eta$ is the parameter of interest, and suppose that the process
$$
Z^{H}(t)=\eta^{1/4} B^{H}(t)+Y(t),\quad t>0,
$$
is sampled at grid points  $t_{i}\in\Upsilon^n(c,d), i=0,1,\ldots,n$, with $[c,d]\subset(0,\infty)$.
Then,
$$
\wh{\eta}_{n}=\frac1{3(d-c)}\sum_{i=1}^{n}\left(Z^{H}(t_{i})-Z^{H}(t_{i-1}) \right)^{4},
$$
is a consistent estimator of $\eta$, since an fBM with Hurst index $H$ has a finite, non-zero  $p=1/H$-variation.
The asymptotic normality of $\mV_{n}^4(B^H;[c,d])$ is established in Theorem~\ref{thm:CLTforSelfSimilarGaussian}, and Corollary~\ref{corollary:sigmaForFBMwithH=1/4},
and hence, by \eqref{eq:CLTSmoothPert2}, $\wh{\eta}_n$ is also asymptotically  normal, and satisfying
\begin{align}\label{eq:CLTforQuarVariationBHT}
\sqrt{n}(\wh{\eta}_{n}-\eta)\xrightarrow[n\to\infty]{\cD}\cN(0,\frac19\check{\sigma}^{2}\eta^{2}),
\end{align}
where $\check{\sigma}^{2}$ is an explicit constant given in Corollary~\ref{corollary:sigmaForFBMwithH=1/4}.

\end{example}

\section{Time sampling at a fixed space point}\label{sec:FixedX}
In this section we assume that the solution $u$ of \eqref{eq:mainSPDE} is measured according to sampling scheme~(A).
We consider the following estimators for $\theta$, and $\sigma^2$ respectively,
\begin{align}
\wh{\theta}_{n,x}& :=\frac{3(d-c)\sigma^{4}}{\pi\sum_{i=1}^{n}(u(t_{i},x)-u(t_{i-1},x))^{4}}, \label{eq:estThetaFixSpace} \\
\wh{\sigma}^2_{n,x}& :=\sqrt{\frac{\theta\pi}{3(d-c)}\sum_{i=1}^{n}(u(t_{i},x)-u(t_{i-1},x))^{4}}. \label{eq:estSigmaFixSpace}
\end{align}
Clearly, \eqref{eq:estThetaFixSpace} assumes that $\sigma$ is known, while \eqref{eq:estSigmaFixSpace} assumes that $\theta$ is known.
We will prove below that these estimators are consistent and asymptotically normal regardless of the nature of the domain on which the equation \eqref{eq:mainSPDE} is considered. We start with the case of bounded domain, Theorem~\ref{th:bounedFixedX}, followed by the whole space, Theorem~\ref{th:fixedSpaceWholeSpace}.

\begin{theorem}\label{th:bounedFixedX}
Let $u$ be the solution to \eqref{eq:mainSPDE} with $G=[0,\pi]$, and assume that $u$ is sampled at discrete points $\set{(t_{i},x) \mid t_i\in\Upsilon^n(c,d)}$, for some fixed $x\in (0,\pi)$, and $0<c<d<\infty$. Then, assuming $\sigma$ is known, $\widehat{\theta}_{n,x}$ given by \eqref{eq:estThetaFixSpace} is a weakly consistent estimator for $\theta$, that is
\begin{equation}\label{eq:weakThetanx1}
\bP\!-\!\lim_{n\to\infty} \widehat{\theta}_{n,x} =\theta.
\end{equation}
Respectively, if $\theta$ is known, then $\wh \sigma^2_{n,x}$ in \eqref{eq:estSigmaFixSpace} is a weakly consistent estimator of $\sigma^2$.  Moreover, $\widehat{\theta}_{n,x}$ and $\wh \sigma^2_{n,x}$ satisfy the following central limit type convergence
\begin{align}
\sqrt{n}\left(\wh{\theta}_{n,x}-\frac{(d-c)\theta}{n\sigma_{n}^{4}} \right)&\xrightarrow[n\to\infty]{\cD}\cN(0,\theta^{2}\left(\bar{\sigma}_{2}^{2}+\bar{\sigma}_{4}^{2}\right))\label{eq:CLTforTheta},\\
\sqrt{n}\left(\wh{\sigma}_{n,x}^{2}-\frac{\sqrt{n}\sigma_{n}^{2}}{\sqrt{d-c}}\sigma^{2} \right)&\xrightarrow[n\to\infty]{\cD}\cN(0,\frac1{36}\sigma^{4}\left(\bar{\sigma}_{2}^{2}+\bar{\sigma}_{4}^{2}\right)),\label{eq:CLTforSigma}
\end{align}
where
\begin{align}
&\sigma_{n}^{2}=\frac{2}{\sqrt{\pi\theta}}\sum_{k\ge 1}\frac{\sin^{2}(kx)}{k^{2}}(1-e^{-(d-c)\theta k^{2}/n}), \label{eq:sigma_n}\\
\bar{\sigma}_{2}^{2}=72+144\lim_{n\rightarrow\infty}&\sum_{j=1}^{n-1}(1-\frac{j}{n})\left|\frac{F(j)}{\sigma_{n}^{2}} \right| ^{2},\quad \bar{\sigma}_{4}^{2}=24+48\lim_{n\rightarrow\infty}\sum_{j=1}^{n-1}(1-\frac{j}{n})\left| \frac{F(j)}{\sigma_{n}^{2}}\right| ^{4}, \label{eq:sigmaBar}
\end{align}
and
\begin{align}
F(j)=\frac1{\sqrt{\pi\theta}}\sum_{k\ge 1}\frac{\sin^{2}(kx)}{k^{2}}\left(2e^{-j(d-c)\theta k^{2}/n}-e^{-(j+1)(d-c)\theta k^{2}/n}-e^{-(j-1)(d-c)\theta k^{2}/n}\right).
\end{align}

\end{theorem}
To study the case of sampling scheme (A) for bounded domain, as it turns out,  is delicate, primarily since there are no ready available convenient representations of the solution, in contrast to the case of whole space discussed later (cf. \eqref{eq:decomSolWithTime}). First we will establish such representation of the solution, which is also an important analytical result on its own.  To the best of our knowledge, the only relevant result regarding this can be found in \cite{walsh1981}, where the author proved that for a similar SPDE at $x=0$ the $4$-variation (in time) of the solution converges to a constant. We will prove that the $4-$variation converges to a constant at any fixed space point $x$. Moreover, we also establish the asymptotic normality property of the 4-variation, for which we use techniques from Malliavin calculus.

\begin{proposition}\label{th:bdDecomTime}
Let $x\in(0,\pi)$ be a fixed space point.  Then, the solution $u(t,x)$ of the equation \eqref{eq:mainSPDE} with $G=[0,\pi]$ admits the following decomposition
\begin{align}\label{eq:u10}
u(t,x)=\frac{\sigma}{(\pi\theta)^{1/4}}v(t)+S(t),\quad t>0,
\end{align}
where $v$ and $S$ are zero-mean Gaussian processes such that:
\begin{enumerate}[(a)]
\item $S(t)$ is continuous on  $[0,+\infty)$, and infinitely differentiable on $(0,\infty)$;
\item $v(t)$ has finite $4-$variation (with convergence in probability)
\begin{align}\label{eq:4varConvInProb}
\bP-\lim_{n\rightarrow\infty}\mV_{n}^{4}(v;[c,d])=3(d-c).
\end{align}
\item the 4-variation admits the asymptotic normality property
\begin{align}\label{eq:CLT4var}
\sqrt{n}\left(\frac{\mV_{n}^{4}(v;[c,d])}{n\sigma_{n}^{4}}-3 \right)\xrightarrow[n\to\infty]{\cD}\cN(0,\bar{\sigma}_{2}^{2}+\bar{\sigma}_{4}^{2}),
\end{align}
where $\sigma_n, \bar{\sigma}_2, \bar{\sigma}_4$ are constants given by \eqref{eq:sigma_n} and \eqref{eq:sigmaBar}.
\end{enumerate}

\end{proposition}

\begin{proof} 

First we note that in this case the Laplace operator $\Delta=\partial_{xx}$ has only discrete spectrum, with eigenvalues $\lambda_k=-k^2, \, k\in\bN$, and with  corresponding eigenfunctions $h_k(x) = \sqrt{2/\pi}\sin(kx), \, k\in\bN$. Moreover, the functions $\set{h_k,\, k\in\bN}$ form a complete orthonormal system in $L^2(G)$, and the noise term can be conveniently written as
$$
W(t,x)=\sum_{k\ge 1}w_{k}(t)h_{k}(x),
$$
where $w_{k}, k\in\bN$, are independent standard Brownian motions. The solution of this equation admits a Fourier series decomposition,
\begin{equation}\label{eq:SolAddBoundedDom}
u(t,x)=\sum_{k\ge 1}u_{k}(t)h_{k}(x),\quad t>0,\quad x\in (0,\pi),
\end{equation}
where each Fourier mode $u_{k}(t)$ is an Ornstein--Uhlenbeck process of the form
\begin{align*}
\dif{u}_{k}(t)&=-\theta k^{2}u_{k}(t)\dif t+\sigma\dif w_{k}(t),\quad t>0,\\
u_{k}(0)&=0.
\end{align*}
Equivalently, we have that
\begin{equation}\label{eq:FourierOU}
u_{k}(t)=\sigma\int_{0}^{t}e^{-\theta k^{2}(t-s)}\dif w_{k}(s).
\end{equation}
Clearly, $u_{k}(t)\sim \mathcal{N}(0,\frac{(1-e^{-2\theta k^{2}t})\sigma^{2}}{2\theta k^{2}})$, and $u_{k}, \, k\in\bN$, are independent random variable.

Assume that $x\in(0,\pi)$ is fixed. We will construct the Gaussian processes $S,v$ explicitly. Let $\{\eta_{k},k\in\bN\}$ be a sequence of i.i.d. standard normal random variables, independent of $\{u_{k},k\in\bN\}$, and let
\begin{align}
S_{k}(t)&:=\frac{\sigma}{\sqrt{2\theta}k}e^{-\theta k^{2}t}\eta_{k},\qquad k\in\bN, \ t\geq 0, \\
S(t)&:=\sum_{k=1}^{\infty}S_{k}(t)h_{k}(x),\quad\ t\geq 0.
\end{align}
Consequently, we put
\begin{align}
v_{k}(t) &:=\frac{(\theta\pi)^{1/4}}{\sigma}\left(u_{k}(t)-S_{k}(t)\right),\qquad k\in\bN, \ t\geq 0, \\
v(t) & :=\sum_{k\ge 1}v_{k}(t)h_{k}(x),\quad t\ge 0,\quad x\in (0,\pi).
\end{align}
Clearly,  $S$ and $v$ are zero-mean Gaussian processes that satisfying \eqref{eq:u10}.

\noindent (a)
It is straightforward to check that $S$ is continuous on $[0,+\infty)]$ and infinitely differentiable on $(0,\infty)$. Moreover,
\begin{align}\label{eq:incrementsk}
\bE \left| S_{k}(t+\epsilon)-S_{k}(t) \right|^{2}=\frac{\sigma^{2}}{2\theta k^{2}}e^{-2\theta k^{2}t}\left(1-e^{-\theta k^{2}\epsilon}\right)^{2}, \qquad k\in\bN, \ t\geq0.
\end{align}

\noindent (b)
By direct computations, using \eqref{eq:FourierOU}, one can show that
\begin{align}
\bE \left| u_{k}(t+\epsilon)-u_{k}(t) \right|^{2}= \frac{\sigma^{2}}{2\theta k^{2}}(1-e^{-\theta k^{2}\epsilon})\left(2-(1-e^{-\theta k^{2}\epsilon})e^{-2\theta k^{2}t}\right),\label{eq:incrementuk}
\end{align}
for $t\geq 0, \ \varepsilon>0, \ k\in\bN$. Combining \eqref{eq:incrementsk}, \eqref{eq:incrementuk} and the independence between $S_{k}$ and $u_{k}$, we deduce that
\begin{align}
\bE\left|v_{k}(t+\epsilon)-v_{k}(t) \right|^{2} =\frac{\sqrt{\pi}}{\sqrt{\theta}k^{2}}(1-e^{-\theta k^{2}\epsilon}), \qquad k\in\bN, \ t\geq 0.
\end{align}
Consequently, we have that
\begin{align}
\bE \left|v(t+\epsilon)-v(t)\right|^{2}=\sum_{k\ge 1}\bE \left|v_{k}(t+\epsilon)-v_{k}(t) \right|^{2}h_{k}^{2}(x)=\frac{2}{\sqrt{\pi\theta}}\sum_{k\ge 1}\frac{\sin^{2}(kx)}{k^{2}}(1-e^{-\theta k^{2}\epsilon}).\label{eq:varofincrement}
\end{align}
We will prove \eqref{eq:4varConvInProb} by showing that
\begin{align}\label{eq:expgoesToCons}
\lim_{n\rightarrow\infty}\bE\left(\mV_{n}^{4}(v;[c,d])\right)&=3(d-c),  \\
\lim_{n\rightarrow\infty}\Var\left(\mV_{n}^{4}(v;[c,d])\right)&=0. \label{eq:vargoesTo0}
\end{align}
Denote by
\begin{align}
\sigma_{n}^{2}:=\bE \left|v(t_{j})-v(t_{j-1})\right|^{2}=\frac{2}{\sqrt{\pi\theta}}\sum_{k\ge 1}\frac{\sin^{2}(kx)}{k^{2}}(1-e^{-(d-c)\theta k^{2}/n}),\quad n\in\bN.
\end{align}
In view of Lemma~\ref{lemma:keylimit},
\begin{align}\label{eq:sigmaAssymp}
\lim_{n\rightarrow\infty}\sqrt{n}\sigma_{n}^{2}=\sqrt{d-c}.
\end{align}
Since $v$ is a zero-mean Gaussian process, we have
$\bE\left|v(t_{j})-v(t_{j-1}) \right|^4=3\sigma_{n}^{4}$,
therefore
$\lim_{n\rightarrow\infty}\bE\left(\mV_{n}^{4}(v;[c,d])\right)=\lim_{n\rightarrow\infty}\sum_{j= 1}^{n}\bE\left|v(t_{j})-v(t_{j-1}) \right|^4=\lim_{n\rightarrow\infty}3n\sigma_{n}^{4}=3(d-c),
$
and hence \eqref{eq:expgoesToCons} is proved.
Next, note that
\begin{align*}
\Var\left(\mV_{n}^{4}(v;[c,d])\right)&=\bE\left(\mV_{n}^{4}(v;[c,d])-\bE\left(\mV_{n}^{4}(v;[c,d])\right)\right)^{2}\\
&=\sum_{j=1}^{n}\bE\left(\left| v(t_{j},x)-v(t_{j-1},x) \right|^4-3\sigma_{n}^{4}\right)^{2}\\
&\qquad + 2\sum_{i<j}\bE\left(\left|v(t_{i},x)-v(t_{i-1},x) \right|^4
-3\sigma_{n}^{4}\right)\left(\left|v(t_{j},x)-v(t_{j-1},x) \right|^4-3\sigma_{n}^{4}\right)\\
&=:J_{1}+J_{2}.
\end{align*}
According to \eqref{eq:sigmaAssymp}, we deduce that
\begin{align}
J_{1}=\sum_{j=1}^{n}\bE\left(\left| v(t_{j},x)-v(t_{j-1},x) \right|^8\right)-9n\sigma_{n}^{8}
=96n\sigma_{n}^{8}\underset{n\to\infty}{\longrightarrow} 0. \label{eq:J1goesTo0}
\end{align}
As far as $J_2$, for $j\ge 1$, we put
\begin{align}\label{eq:F}
F(j)&:=\bE\left(v(t_{i},x)-v(t_{i-1},x)\right)\left(v(t_{i+j},x)-v(t_{i+j-1},x)\right)\\
&=\frac1{\sqrt{\pi\theta}}\sum_{k\ge 1}\frac{\sin^{2}(kx)}{k^{2}}\left(2e^{-j(d-c)\theta k^{2}/n}-e^{-(j+1)(d-c)\theta k^{2}/n}-e^{-(j-1)(d-c)\theta k^{2}/n}\right) \\
&=G_{j}-G_{j-1},
\end{align}
where
\begin{align}
G_{j}:=\frac1{\sqrt{\pi\theta}}\sum_{k\ge 1}\frac{\sin^{2}(kx)}{k^{2}}\left(e^{-j(d-c)\theta k^{2}/n}-e^{-(j+1)(d-c)\theta k^{2}/n}\right),\quad j\geq 0,
\end{align}
and also put $F(0):=\sigma_{n}^{2}$. Since $F(j)<0$, we have that $G_{j}<G_{j-1}$.
Using the property of joint normal distributions, we continue
\begin{align}
J_{2}&=2\sum_{i<j}\bE\left(\left| v(t_{i},x)-v(t_{i-1},x) \right|^4-3\sigma_{n}^{4}\right)\left(\left| v(t_{j},x)-v(t_{j-1},x) \right|^4-3\sigma_{n}^{4}\right)\\
&=2\sum_{i<j}\left(24F^{4}(j-i)+72F^{2}(j-i)\sigma_{n}^{4} \right).
\end{align}
From here, since $\left|F(j-i)\right|\leq \sigma_{n}^{2}$, we deduce that
\begin{align}
J_{2}&\leq 2\sum_{i<j}\left(24|F(j-i)|\sigma_{n}^{6}+72|F(j-i)|\sigma_{n}^{6} \right)=192\sum_{i<j}|F(j-i)|\sigma_{n}^{6}\\
&=192\sigma_{n}^{6}\sum_{j=1}^{n-1}(n-j)\left(G_{j-1}-G_{j}\right).
\end{align}
Note that $\sum_{j=1}^{n-1}(n-j)\left(G_{j-1}-G_{j}\right)=nG_{0}-\sum_{j=0}^{n-1}G_{j}$, and since
\begin{align}
\sum_{j=0}^{n-1}G_{j}&=\sum_{j=0}^{n-1}\frac1{\sqrt{\pi\theta}}\sum_{k\ge 1}\frac{\sin^{2}(kx)}{k^{2}}\left(e^{-j(d-c)\theta k^{2}/n}-e^{-(j+1)(d-c)\theta k^{2}/n}\right)\\
&=\frac1{\sqrt{\pi\theta}}\sum_{k\ge 1}\frac{\sin^{2}(kx)}{k^{2}}\left(1-e^{-(d-c)\theta k^{2}}\right)=\frac12\sigma_{1}^{2},
\end{align}
and $G_{0}=\frac12\sigma_{n}^{2}$, we conclude that
\begin{align}
J_{2}&\leq 192\sigma_{n}^{6}\left(n\frac1{\sqrt{\pi\theta}}\sum_{k\ge 1}\frac{\sin^{2}(kx)}{k^{2}}\left(1-e^{-(d-c)\theta k^{2}/n}\right)-\frac1{\sqrt{\pi\theta}}\sum_{k\ge 1}\frac{\sin^{2}(kx)}{k^{2}}\left(1-e^{-(d-c)\theta k^{2}}\right) \right)\\
&=192\sigma_{n}^{6}\left(\frac{n}{2}\sigma_{n}^{2}-\frac12\sigma_{1}^{2} \right)\overset{n\to\infty}{\longrightarrow} 0. \label{eq:J2goesTo0}
\end{align}
Combining \eqref{eq:J1goesTo0} and \eqref{eq:J2goesTo0}, \eqref{eq:vargoesTo0} is proved.
Consequently, by \eqref{eq:expgoesToCons} and \eqref{eq:vargoesTo0}, we also have that
$\mV_{n}^{4}(v;[c,d])$ converges to $3(d-c)$, both in $L^{2}$ and in probability.

\medskip
\noindent (c) At general level, the proof  of \eqref{eq:CLT4var} is in line with the proof of the central limit theorem in \cite{Corcuera2012} established for a similar but much simpler covariance structure. More precisely, we will apply Theorem~\ref{th:malliavin4thMomentThm}, by showing that   \eqref{eq:NO081} and condition (N1) are satisfied.
We begin by showing that
\begin{align}\label{eq:covarSeries}
\sum_{j=-l}^{r}\left|F(|j|) \right|^{m}\le 2\sigma_{n}^{2m},
\end{align}
for any $m\ge 1$, $\ell,r\in\bN$.
Since $m\ge 1$,
\begin{align}
\sum_{j=1}^{r}\left| F(j) \right|^{m} &= \sum_{j=1}^{r}\left|F(j) \right|^{m-1}\left|F(j) \right|
\leq \sum_{j=1}^{r}\sigma_{n}^{2(m-1)} |F(j)|\\
&=\sum_{j=1}^{r}\sigma_{n}^{2(m-1)}(G_{j-1}-G_{j})=\sigma_{n}^{2(m-1)}(G_{0}-G_{r-1})\\
&\leq \sigma_{n}^{2(m-1)}G_{0} =\frac12\sigma_{n}^{2m},
\end{align}
where we used the fact that $G_j\geq 0$ and $G_{0}=\frac12\sigma_{n}^{2}$.
Therefore,
\begin{align}
\sum_{j=-l}^{r}\left|F(|j|) \right|^{m} &= (\sigma_{n}^{2})^{m} + \sum_{j=1}^{r} |F(j)|^{m} + \sum_{j=1}^{l}|F(j)|^{m}\\
&\leq \sigma_{n}^{2m} + \frac12\sigma_{n}^{2m} + \frac12\sigma_{n}^{2m} = 2\sigma_{n}^{2m}.
\end{align}
With slight abuse of notations, just in this proof, we denote by $\Delta v_{j}^{n}:=v(t_{j},x) - v(t_{j-1},x)$.
Let $\cH$ be the closed subspace of $L^{2}(\Omega, \cF, \bP)$ generated by the random variables
$\frac{\Delta v_{j}^{n}}{\sigma_{n}}$, $1\leq j\leq n; \ j,n\in\bN$.
Then,
\begin{align}
\left| \frac{\Delta v_{j}^{n}}{\sigma_{n}} \right|^{4} - 3 & =\left( \left| \frac{\Delta v_{j}^{n}}{\sigma_{n}} \right|^{4} - 6\left| \frac{\Delta v_{j}^{n}}{\sigma_{n}} \right|^{2}+3\right)+6\left(\left| \frac{\Delta v_{j}^{n}}{\sigma_{n}} \right|^{2}-1\right)\\
&=H_{4}\left(\frac{\Delta v_{j}^{n}}{\sigma_{n}}\right)+6H_{2}\left(\frac{\Delta v_{j}^{n}}{\sigma_{n}}\right)=I_{4}\left[\left(\frac{\Delta v_{j}^{n}}{\sigma_{n}}\right)^{\otimes4} \right]+6I_{2}\left[\left(\frac{\Delta v_{j}^{n}}{\sigma_{n}}\right)^{\otimes2} \right].
\end{align}
Therefore,
\begin{align}\label{eq:QVarAsWienerIntegral}
\sqrt{n}\left(\frac{\mV_{n}^{4}(v;[c,d])}{n\sigma_{n}^{4}}-3\right) =I_{4}\left[\frac1{\sqrt{n}}\sum_{j=1}^{n}\left(\frac{\Delta v_{j}^{n}}{\sigma_{n}}\right)^{\otimes4} \right]+I_{2}\left[\frac6{\sqrt{n}}\sum_{j= 1}^{n}\left(\frac{\Delta v_{j}^{n}}{\sigma_{n}}\right)^{\otimes2} \right].
\end{align}
Let
\begin{align}\label{eq:defOf_fn2_fn4}
f_n^{(2)} := \frac{6}{\sqrt{n}}\sum_{j=1}^{n}\left(\frac{\Delta v_{j}^{n}}{\sigma_{n}}\right)^{\otimes2}, \qquad f_n^{(4)}:=\frac1{\sqrt{n}}\sum_{j=1}^{n}\left(\frac{\Delta v_{j}^{n}}{\sigma_{n}}\right)^{\otimes4},
\end{align}
and consider the sequence of two dimensional random vectors $F_{n} := \left(I_{2}(f_{n}^{(2)}),I_{4}(f_{n}^{(4)}) \right), n\in\bN$, to which we will apply Theorem~\ref{th:malliavin4thMomentThm}.
Using the properties of Wiener integral, we obtain that
\begin{align}
\lim_{n\rightarrow\infty}\bE\left(I_{2}(f_{n}^{(2)})I_{4}(f_{n}^{(4)}) \right)=0,
\end{align}
and hence \eqref{eq:NO081} is satisfied.

Next, we move to verification of condition (N1), which in this case becomes
\begin{align}\label{eq:condition3}
\lim_{n\rightarrow\infty}\norm {f_{n}^{(m)}\otimes_{r} f_{n}^{(m)}} _{H^{2\otimes(m-r)}}^{2}=0.
\end{align}
for $m=2,4$, and $1\leq r\leq m-1$.

Using the linearity of the inner products and the properties of the tensor products of Hilbert spaces, we obtain
\begin{align}
\bE\left(I_{2}(f_{n}^{(2)}) \right)^{2}&= 2\langle f_{n}^{(2)},f_{n}^{(2)}\rangle_{\cH^{\otimes2}}
=\frac{72}{n}\Big\langle\sum_{j=1}^{n}\left(\frac{\Delta v_{j}^{n}}{\sigma_{n}}\right)^{\otimes2},\sum_{j=1}^{n}\left(\frac{\Delta v_{j}^{n}}{\sigma_{n}}\right)^{\otimes2}\Big\rangle_{\cH^{\otimes2}}\\
&=\frac{72}{n}\sum_{i,j=1}^{n}\Big\langle\left(\frac{\Delta v_{i}^{n}}{\sigma_{n}}\right)^{\otimes2},\left(\frac{\Delta v_{j}^{n}}{\sigma_{n}}\right)^{\otimes2}\Big\rangle_{\cH^{\otimes2}}
=\frac{72}{n}\sum_{i,j=1}^{n}\Big\langle\frac{\Delta v_{i}^{n}}{\sigma_{n}},\frac{\Delta v_{j}^{n}}{\sigma_{n}}\Big\rangle_{\cH}^{2}\\
&=\frac{72}{n}\sum_{i,j=1}^{n}\left[\bE\left(\frac{\Delta v_{i}^{n}}{\sigma_{n}}\cdot\frac{\Delta v_{j}^{n}}{\sigma_{n}}\right)\right]^{2}=\frac{72}{n}\sum_{i,j=1}^{n}\frac{\left|F(|j-i|) \right|^{2}}{\sigma_{n}^{4}}\\
&=\frac{72}{n\sigma_{n}^{4}}\left(\sum_{j=1}^{n}|F(0)|^{2}+2\sum_{i<j}|F(j-i)|^{2} \right)=\frac{72}{n\sigma_{n}^{4}}\left(n\sigma_{n}^{4}+2\sum_{j=1}^{n-1}(n-j)|F(j)|^{2} \right)\\
&=72+\frac{144}{\sigma_{n}^{4}}\sum_{j=1}^{n-1}(1-\frac{j}{n}) |F(j)|^{2}=72+144\sum_{j=1}^{n-1}(1-\frac{j}{n})\left|\frac{F(j)}{\sigma_{n}^{2}}\right|^{2},
\end{align}
In view of \eqref{eq:covarSeries}, we have that
$$
\sum_{j=1}^{n-1}(1-\frac{j}{n})\left| \frac{F(j)}{\sigma_{n}^{2}}\right|^{2}
\leq \sum_{j=1}^{\infty}\left| \frac{F(j)}{\sigma_{n}^{2}}\right|^{2} <\infty,
$$
and thus
\begin{align}\label{eq:sigmabar2square}
\bar{\sigma}_{2}^{2}:=\lim_{n\rightarrow\infty}\bE\left(I_{2}(f_{n}^{(2)}) \right)^{2}=72+144\lim_{n\rightarrow\infty}\sum_{j=1}^{n-1}(1-\frac{j}{n})
\left| \frac{F(j)}{\sigma_{n}^{2}}\right|^{2} <\infty.
\end{align}
Similarly,
\begin{align}
\bE\left(I_{4}(f_{n}^{(4)}) \right)^{2} &= 24\Big\langle f_{n}^{(4)},f_{n}^{(4)}\Big\rangle_{\cH^{\otimes4}}
=\frac{24}{n}\Big\langle\sum_{j=1}^{n}\left(\frac{\Delta v_{j}^{n}}{\sigma_{n}}\right)^{\otimes4},\sum_{j=1}^{n}\left(\frac{\Delta v_{j}^{n}}{\sigma_{n}}\right)^{\otimes4}\Big\rangle_{\cH^{\otimes4}}\\
&=\frac{24}{n}\sum_{i,j=1}^{n}\Big\langle\left(\frac{\Delta v_{i}^{n}}{\sigma_{n}}\right)^{\otimes4},\left(\frac{\Delta v_{j}^{n}}{\sigma_{n}}\right)^{\otimes4}\Big\rangle_{\cH^{\otimes4}}=\frac{24}{n}\sum_{i,j=1}^{n}\Big\langle\frac{\Delta v_{i}^{n}}{\sigma_{n}},\frac{\Delta v_{j}^{n}}{\sigma_{n}}\Big\rangle_{\cH}^{4}\\
&=\frac{24}{n}\sum_{i,j=1}^{n}\left[\bE\left(\frac{\Delta v_{i}^{n}}{\sigma_{n}}\cdot\frac{\Delta v_{j}^{n}}{\sigma_{n}}\right)\right]^{4}
=\frac{24}{n}\sum_{i,j=1}^{n}\frac{\left| F(|j-i|) \right|^{4}}{\sigma_{n}^{8}}\\
&\leq 24+48\sum_{j=1}^{n-1}\left|\frac{F(j)}{\sigma_{n}^{2}}\right|^{4}
\leq 24+48\sum_{j=1}^{\infty}\left|\frac{F(j)}{\sigma_{n}^{2}}\right|^{4} <\
\infty,
\end{align}
and consequently,
\begin{align}\label{eq:sigmabar4square}
\bar{\sigma}_{4}^{2}:=\lim_{n\rightarrow\infty}\bE\left(I_{4}(f_{n}^{(4)}) \right)^{2}
=24+48\lim_{n\rightarrow\infty}\sum_{j=1}^{n-1}(1-\frac{j}{n}) \left|\frac{F(j)}{\sigma_{n}^{2}}\right|^{4}<\infty.
\end{align}
Let $a_{2}=6,a_{4}=1$. Then,
\begin{align}
\norm {f_{n}^{(m)}\otimes_{r} f_{n}^{(m)}} _{H^{2\otimes(m-r)}}^{2}&=\norm {\frac{a_{m}}{\sqrt{n}}\sum_{j=1}^{n}\left(\frac{\Delta v_{j}^{n}}{\sigma_{n}}\right)^{\otimes m}\otimes_{r} \frac{a_{m}}{\sqrt{n}}\sum_{j=1}^{n}\left(\frac{\Delta v_{j}^{n}}{\sigma_{n}}\right)^{\otimes m}} _{H^{\otimes2(m-r)}}^{2}\\
&=\norm {\frac{a_{m}^{2}}{n}\sum_{i,j=1}^{n}\left(\frac{\Delta v_{i}^{n}}{\sigma_{n}}\right)^{\otimes m}\otimes_{r} \left(\frac{\Delta v_{j}^{n}}{\sigma_{n}}\right)^{\otimes m}} _{H^{\otimes2(m-r)}}^{2}\\
&=\norm {\frac{a_{m}^{2}}{n}\sum_{i,j=1}^{n}\big\langle \frac{\Delta v_{i}^{n}}{\sigma_{n}},\frac{\Delta v_{j}^{n}}{\sigma_{n}} \big\rangle_{H}^{r}\left(\frac{\Delta v_{i}^{n}}{\sigma_{n}}\right)^{\otimes (m-r)}\otimes \left(\frac{\Delta v_{j}^{n}}{\sigma_{n}}\right)^{\otimes (m-r)}} _{H^{\otimes2(m-r)}}^{2}\\
&=\norm {\frac{a_{m}^{2}}{n}\sum_{i,j=1}^{n}\frac{ \left| F(|j-i|) \right|^{r}}{\sigma_{n}^{2r}}\left(\frac{\Delta v_{i}^{n}}{\sigma_{n}}\right)^{\otimes (m-r)}\otimes \left(\frac{\Delta v_{j}^{n}}{\sigma_{n}}\right)^{\otimes (m-r)}} _{H^{\otimes2(m-r)}}^{2}\\
&=\frac{a_{m}^{4}}{n^{2}\sigma_{n}^{4m}}\sum_{i,j,i',j'=1}^{n}|F(|j-i|)|^{r}|F(|j'-i'|)|^{r}|F(|i'-i|)|^{m-r}|F(|j'-j|)|^{m-r}\\
&\leq \frac{a_{m}^{4}}{n^{2}\sigma_{n}^{4m}}\sum_{i,j,i',j'=1}^{n}\left| F(|j-i|)F(|j'-i'|)F(|i'-i|)F(|j'-j|) \right|\sigma_{n}^{4m-8}\\
&=\frac{a_{m}^{4}}{n^{2}\sigma_{n}^{8}}\sum_{i,j,i',j'=1}^{n}\left|F(|j-i|)F(|j'-i'|)F(|i'-i|)F(|j'-j|) \right|\\
&= O_{1}+2O_{2},
\end{align}
where
\begin{align}
O_{1}&:=\frac{a_{m}^{4}}{n^{2}\sigma_{n}^{8}}\sum_{i',j'=1}^{n}\sum_{i=1}^{n}\Mid F(0)F(|j'-i'|)F(|i'-i|)F(|j'-i|) \Mid,\\
O_{2}&:=\frac{a_{m}^{4}}{n^{2}\sigma_{n}^{8}}\sum_{i',j'=1}^{n}\sum_{i<j}\Mid F(|j-i|)F(|j'-i'|)F(|i'-i|)F(|j'-j|) \Mid.
\end{align}
Note that, by direct computations and using \eqref{eq:covarSeries}, we have
\begin{align}
O_{1}&=\frac{a_{m}^{4}}{n^{2}\sigma_{n}^{6}}\sum_{i',j'=1}^{n}\sum_{i=1}^{n}\Mid F(|j'-i'|)F(|i'-i|)F(|j'-i|) \Mid\\
&\leq \frac{a_{m}^{4}}{n^{2}\sigma_{n}^{6}}\sum_{i',j'=1}^{n}\sum_{i=1}^{n}\mid F(|j'-i'|)\mid \frac{F(|i'-i|)^{2}+F(|j'-i|)^{2}}{2}\\
&\leq \frac{a_{m}^{4}}{n^{2}\sigma_{n}^{6}}\sum_{i',j'=1}^{n}\mid F(|j'-i'|)\mid \frac{2\sigma_{n}^{4}+2\sigma_{n}^{4}}{2} \leq \frac{2a_{m}^{4}}{n^{2}\sigma_{n}^{2}}\sum_{i',j'=1}^{n}\mid F(|j'-i'|)\mid \\
&\leq \frac{2a_{m}^{4}}{n^{2}\sigma_{n}^{2}}\left(\sum_{j=1}^{n}|F(0)|+2\sum_{i<j}\mid F(j-i)\mid \right)\\
&\leq \frac{2a_{m}^{4}}{n}+\frac{4a_{m}^{4}}{n^{2}\sigma_{n}^{2}}\sum_{j=1}^{n-1}(n-j)\mid F(j)\mid = \frac{2a_{m}^{4}}{n}+\frac{4a_{m}^{4}}{n}\sum_{j=1}^{n-1}(1-\frac{j}{n})\mid \frac{F(j)}{\sigma_{n}^{2}}\mid\\
&\underset{n\to\infty}{\longrightarrow} 0.
\end{align}
Similarly,
\begin{align}
O_{2}&=\frac{a_{m}^{4}}{n^{2}\sigma_{n}^{8}}\sum_{i',j'=1}^{n}\sum_{i=1}^{n-1}\sum_{k=1}^{n-i}\Mid F(|i+k-i|)F(|j'-i'|)F(|i'-i|)F(|j'-i-k|) \Mid\\
&=\frac{a_{m}^{4}}{n^{2}\sigma_{n}^{8}}\sum_{i',j'=1}^{n}\sum_{i=1}^{n-1}\sum_{k=1}^{n-i}\Mid F(k)F(|j'-i'|)F(|i'-i|)F(|j'-i-k|) \Mid\\
&\leq \frac{a_{m}^{4}}{n^{2}\sigma_{n}^{8}}\sum_{i',j'=1}^{n}\sum_{i=1}^{n-1}\sum_{k=1}^{n-i}\Mid F(|j'-i'|)F(|i'-i|)\Mid \frac{F(k)^{2}+F(|j'-i-k|)^{2}}{2}\\
&\leq \frac{2a_{m}^{4}}{n^{2}\sigma_{n}^{4}}\sum_{i',j'=1}^{n}\sum_{i=1}^{n-1}\Mid F(|j'-i'|)F(|i'-i|)\Mid\leq \frac{4a_{m}^{4}}{n^{2}\sigma_{n}^{2}}\sum_{i',j'=1}^{n}\Mid F(|j'-i'|)\Mid\\
&\underset{n\to\infty}{\longrightarrow} 0.
\end{align}
Thus, \eqref{eq:condition3} holds true. Therefore, (N2) from  Theorem~\ref{th:malliavin4thMomentThm} holds true, namely, we have that
\begin{align}\label{eq:CLTF}
F_{n}\xrightarrow[n\to\infty]{\cD}\cN\left(0,\left(\begin{array}{lll}
&\bar{\sigma}_{2}^{2} & 0\\
&0 &\bar{\sigma}_{4}^{2}
\end{array}\right)\right). \end{align}
Consequently, \eqref{eq:CLT4var} follows from \eqref{eq:QVarAsWienerIntegral}, \eqref{eq:defOf_fn2_fn4} and \eqref{eq:CLTF}.

The proof is complete.

\end{proof}

\begin{proof}[Proof of Theorem~\ref{th:bounedFixedX}]
By Proposition~\ref{prop:QVarSmoothPert} and Proposition~\ref{th:bdDecomTime}.(a)-(b) we have that
$$
\bP-\lim_{n\to\infty}\mV_n^4(u(\cdot,x); [c,d]) = \frac{3(d-c)\sigma^4}{\pi\theta},
$$
which implies consistency of $\wh{\theta}_{n,x}$ and $\wh{\sigma}_{n,x}^2$.

By similar arguments as in the proof of Proposition~\ref{prop:QVarSmoothPert}, one can also show that
\begin{align}\label{eq:inter1}
\sqrt{n}\left(\frac{\pi\theta\mV_{n}^{4}\left(u(\cdot,x);[c,d]\right)}{n\sigma_{n}^{4}\sigma^{4}}-\frac{\mV_{n}^{4}(v;[c,d])}{n\sigma_{n}^{4}} \right)\rightarrow 0,\quad \textrm{ in }L^{2}\textrm{ and in probability.}
\end{align}
By using this, and \eqref{eq:CLT4var} we have that
\begin{align}\label{eq:CLTforU}
\sqrt{n}\left(\frac{\pi\theta\mV_{n}^{4}\left(u(\cdot,x);[c,d]\right)}{n\sigma_{n}^{4}\sigma^{4}}-3 \right)\xrightarrow[n\to\infty]{\cD}\cN(0,\bar{\sigma}_{2}^{2}+\bar{\sigma}_{4}^{2}).
\end{align}	
Combining \eqref{eq:estThetaFixSpace} and \eqref{eq:CLTforU}, we have
\begin{align}\label{eq:1stCLT}
\sqrt{n}\left(\frac{3(d-c)\theta}{\wh{\theta}_{n,x}n\sigma_{n}^{4}}-3\right)\xrightarrow[n\to\infty]{\cD}\cN(0,\bar{\sigma}_{2}^{2}+\bar{\sigma}_{4}^{2}).
\end{align}
Finally, due to \eqref{eq:weakThetanx1}, and by Slutsky's theorem, \eqref{eq:CLTforTheta} follows at once. Relationship \eqref{eq:CLTforSigma} is proved similarly. This completes the proof.
\end{proof}

\begin{theorem}\label{th:fixedSpaceWholeSpace}
Let $u$ be the solution to \eqref{eq:mainSPDE} with $G=\bR$, and assume that $u$ is sampled at discrete points $\set{(t_{i},x) \mid t_i\in\Upsilon^n(c,d)}$, for some fixed $x\in \bR$, and $0<c<d<\infty$.
Assuming that $\sigma$ is known, we have that $\wh\theta_{n,x}$ is (strongly) consistent and asymptotically normal estimator of $\theta$, i.e.
\begin{align}
& \lim_{n\rightarrow\infty}\wh{\theta}_{n,x}=\theta,\quad\bP-a.s. \label{eq:consistTheta4} \\
& \sqrt{n}(\wh{\theta}_{n,x}-\theta)\xrightarrow[n\to\infty]{\cD}\cN(0,\frac19\theta^{2} \check{\sigma}^{2}), \quad l=2,4.
\label{eq:asynormThetaFixSpace}
\end{align}
where
\begin{equation}\label{eq:constCCheck}
\check{\sigma}^{2}=72\check{\sigma}_{2}^{2}+24\check{\sigma}_{4}^{2}, \qquad
\check{\sigma}_{l}^{2}=\lim_{n\rightarrow\infty}\frac1{n}\sum_{i=1}^{n}\sum_{j=1}^{n}r^{l}(|i-j|).
\end{equation}
Accordingly, assuming that $\theta$ is known, we have that
\begin{align}
&\lim_{n\rightarrow\infty}\wh{\sigma}^{2}_{n,x}=\sigma^{2},\quad \bP-a.s. \label{eq:consistsigma4} \\
&\sqrt{n}(\wh{\sigma}^{2}_{n,x}-\sigma^{2})\xrightarrow[n\to\infty]{\cD}\cN(0,\frac1{36}\sigma^{4}\check{\sigma}^{2}).\label{eq:asynormSigmaFixSpace}
\end{align}
\end{theorem}
\begin{proof}
We will use the following representations (cf. \cite[Section~3]{Khoshnevisan2014BookSBMS}) of the solution $u$ of  \eqref{eq:mainSPDE} when $G=\bR$. For every fixed $x\in\bR$, there exists a fractional Brownian motion $B^{H}(t)$ with Hurst index $H=1/4$ and a Gaussian process $Y(t)$ that is continuous on $\bR_{+}$ and infinitely differentiable on $(0,\infty)$, such that
\begin{align}\label{eq:decomSolWithTime}
u(t,x)=\frac{\sigma}{(\theta\pi)^{1/4}}B^{H}(t)+Y(t),\quad t>0.
\end{align}
With this at hand, we apply the results from Example~\ref{ex:fBM} and \eqref{eq:consistTheta4}, \eqref{eq:asynormThetaFixSpace}, \eqref{eq:consistsigma4} follows easily. In addition, applying Delta-method, relationship \eqref{eq:asynormSigmaFixSpace} also follows at once. This concludes the proof.
\end{proof}

\section{Space sampling at a fixed time instance}\label{sec:fixedTime}
Assume that $t>0$ is a fixed time instant, and consider the partition $\Upsilon^m(a,b)$ of the fixed interval $[a,b]\subset G$. Suppose that the solution $u$ of \eqref{eq:mainSPDE} is observed at the grid points $\set{(t,x_{j})\mid  x_j\in \Upsilon^m(a,b), j=1,\ldots,m}$.
Consider the following estimators for $\theta$ and $\sigma^2$ respectively
\begin{align}
\wt{\theta}_{m,t} & :=\frac{(b-a)\sigma^{2}}{2\sum_{j=1}^{m}(u(t,x_{j})-u(t,x_{j-1}))^{2}}, \label{eq:estThetaFixTime} \\
\wt{\sigma}^{2}_{m,t}& :=\frac{2\theta}{b-a}\sum_{j=1}^{m}(u(t,x_{j})-u(t,x_{j-1}))^{2}. \label{eq:estSigmaFixTime}
\end{align}
Similar to Section~\ref{sec:FixedX}, estimator \eqref{eq:estThetaFixTime} assumes that $\sigma$ is known, while \eqref{eq:estSigmaFixTime} assumes that $\theta$ is known. Next we present the main result of this section, that shows that these estimators are consistent and asymptotically normal.

\begin{theorem}\label{th:add1}
Assume that $u$ is the solution of \eqref{eq:mainSPDE} with $G=[0,\pi]$ or $\bR$, and suppose that $u$ is observed according to sampling scheme (B).
Assuming that $\sigma$ is known, the estimator \eqref{eq:estThetaFixTime} of $\theta$ is (strongly) consistent, i.e. $\lim_{m\rightarrow\infty}\wt{\theta}_{m,t}=\theta$ with probability one, and asymptotically normal,
\begin{align}\label{eq:asynormThetaFixTime}
\sqrt{m}(\wt{\theta}_{m,t}-\theta)\xrightarrow[m\to\infty]{\cD}\cN(0,2\theta^{2}).
\end{align}
Assuming that $\theta$ is known, the estimator \eqref{eq:estSigmaFixTime} is a (strongly) consistent and asymptotically normal estimator of $\sigma^2$, with
\begin{align}\label{eq:estSigmaAddFixTime}
\sqrt{m}(\wt{\sigma}^{2}_{m,t}-\sigma^{2})\xrightarrow[m\to\infty]{\cD}\cN(0,2\sigma^{4}).
\end{align}
\end{theorem}

\begin{proof} We begin with the case of bounded domain $G=[0,\pi]$. Recall that $u$ in this case is given by \eqref{eq:SolAddBoundedDom}. We will show that\footnote{A similar result, left as an exercise,  can be found in \cite[Exercise~3.10]{Walsh}.} for every fixed $t>0$, there is a Brownian motion $B(x)$ on $[0,\pi]$, and   a Gaussian process $R(x), \, x\in[0,\pi]$ with a $C^\infty(0,\pi)$ version, such that
\begin{align}\label{eq:decompSolWithSpace}
u(t,x)= \frac{\sigma}{\sqrt{2\theta}}B(x) + R(x), \quad x\in[0,\pi].
\end{align}
Indeed, it is enough to take
\begin{align*}
B(x)  &= \xi_{0}+ \sum_{k\ge 1}\frac1{k}\xi_{k}h_{k}(x),
&R(x)  &= -\frac{\sigma x}{\sqrt{2\theta\pi}}\xi_{0}+\frac{\sigma}{\sqrt{2\theta}}\sum_{k\ge 1}\frac{a_{k}-1}{k}\xi_{k}h_{k}(x), \\
 \xi_{k}&=\sqrt{\frac{2\theta k^{2}}{(1-e^{-2\theta k^{2}t})\sigma^{2}}}u_{k}(t),
& a_{k}&=\sqrt{1-e^{-2\theta k^{2}t}}.
\end{align*}
Note that $\xi_k$ are i.i.d. standard Gaussian random variables. It is easy to check that $B$ is a standard Brownian motion on $[0,\pi]$, for example by noting that $B$ is the  Karhunen--Lo\`eve expansion for the Brownian motion, up to some change of variables.
It is also straightforward to show that $R$ is smooth.

With the representation  \eqref{eq:decompSolWithSpace} at hand, in view of Proposition~\ref{prop:QVarSmoothPert} and Example~\ref{ex:BM}, consistency of $\wt{\theta}_{m,t}$ and  $\wt\sigma_{n,t}^2$, as well as asymptotic normality of $\wt\sigma_{n,t}^2$ follows at once.  In addition, employing the Delta-method, also yields \eqref{eq:asynormThetaFixTime}.

The case of whole space $G=\bR$ is addressed similarly. In view of \cite[Section~3]{Khoshnevisan2014BookSBMS}, the decomposition \eqref{eq:decompSolWithSpace} also holds true in this case, with $B(x)$ being a two-sided Brownian and $X(x)$ being a Gaussian process $X(x)$ with a $C^{\infty}(\bR)$ version.

This concludes the proof.  \end{proof}

\section{Space-time sampling and joint estimation of $\theta$ and $\sigma$}\label{sec:space-time}
While the main goal of this work is to find estimators for drift $\theta$ and volatility $\sigma$ assuming minimal information, and also to prove their asymptotic properties, in this section we will address several practical questions related to this problem.

For both sampling schemes (A) and (B), we assumed that one of the two parameters $\theta$ and $\sigma$ can be consistently estimated, if the other one is known.
The first natural question is how to estimate $\theta$ and $\sigma$ simultaneously. For this, it is enough to observe the solution once according to sampling scheme (A) and once by sampling scheme (B).  Indeed, the key observation is that by sampling scheme (A) one can estimate consistently the ratio $\sigma^4/\theta$, while sampling scheme (B) yields a consistent estimator of $\sigma^2/\theta$. Hence, in view of Theorem~\ref{th:bounedFixedX}, Theorem~\ref{th:fixedSpaceWholeSpace} and Theorem~\ref{th:add1}, we have the following consistent estimators for $\theta$ and $\sigma$
\begin{equation}\label{eq:Simultan}
\begin{aligned}
\bar{\theta}_{n,m}&:=\frac{\pi (b-a)^2  \mV_n^4(u(\cdot,x); [c,d]) }
{12(d-c)(\mV_m^2(u(t,\cdot; [a,b])))^2} \xrightarrow[n,m\to\infty]{} \theta  \\
\bar{\sigma}_{n,m}^2&:=\frac{\pi(b-a) \mV_n^4(u(\cdot,x); [c,d])}{6 (d-c) \mV^2_m(u(t,\cdot);[a,b])}\xrightarrow[n,m\to\infty]{} \sigma^2,
\end{aligned}
\end{equation}
where the convergence is either in probability or a.s.

Next, we also consider the estimation problem of $\theta$ and $\sigma$ when the solution $u(t,x)$ is sampled on discrete space-time grid $(t_i,x_j)$, $t_i\in\Upsilon^n(a,b), \ x_j\in\Upsilon^m(c,d)$. Similar to \cite{BibingerTrabs2017}, we simply take the average of the previous estimators with respect to other dimension. Namely, we put
\begin{equation}\label{eq:sampleSpaceTime}
\begin{aligned}
 \wh\theta_{(n,m)} & : = \frac{1}{m} \sum_{j=1}^{m} \wh\theta_{n,x_j}, &
 \wt\theta_{(n,m)} & : = \frac{1}{n} \sum_{i=1}^{n} \wt\theta_{m,t_i}, \\
 \wh\sigma_{(n,m)}^2 & : = \frac{1}{m} \sum_{j=1}^{m} \wh\sigma^2_{n,x_j}, &
 \wt\sigma_{(n,m)}^2 & : = \frac{1}{n} \sum_{i=1}^{n} \wt\sigma^2_{m,t_i}.
  \end{aligned}
  \end{equation}
The consistency of these estimators follows from the results of Section~\ref{sec:FixedX} and Section~\ref{sec:fixedTime}. Similar estimators can be constructed by using \eqref{eq:Simultan}.  The asymptotic normality of the estimators, as well as of $\bar{\theta}_{n,m}$ and $\bar{\sigma}_{n,m}$, is more intricate due to highly nontrivial covariance structure associated with these estimators. This remains an open problem and it will be investigated by the authors in the future works. We conjecture that all estimators in \eqref{eq:sampleSpaceTime} exhibit a rate of convergence equal to $\sqrt{nm}$.

\section{Numerical examples}\label{sec:numerical}
In this section we will present an illustrative numerical example for the main theoretical results. We consider the stochastic heat equation \eqref{eq:mainSPDE} on interval $[0,\pi]$, with zero boundary values, and zero initial conditions. We use the Fourier decomposition \eqref{eq:SolAddBoundedDom} to approximate numerically the solution $u$ of \eqref{eq:mainSPDE}, by fixing $\theta=0.1$, $\sigma=0.2$ and using 15000 Fourier modes. Each Fourier mode is simulated by using exponential Euler scheme, on the same time grid, $t_0=0, t_1, \ldots, t_n=1$.

First we focus on space sampling results, Theorem~\ref{th:add1}. Assuming that $\theta$ is the parameter of interest, we use  \eqref{eq:estThetaFixTime}, to estimate it at some fixed time points, and by taking $[a,b]=[0,\pi]$. A sample path of the estimator $\wt\theta_{m,t}$, are presented in Figure~\ref{fig1}, left panel\footnote{For all figures in this paper, the left panel is dedicated to $\theta$ and the right panel is dedicated to $\sigma$.}, for $t=0.4$ and $t=1$. As expected, the estimator $\wt\theta_{m,t}$ converges to the true value as number of points in partition $\Upsilon^m(0,\pi)$ increases. In Figure~\ref{fig2} we display the sample mean of the $\wt\theta_{m,t}$ computed from 1000 Monte Carlo simulations, which also converges to the true value. In Figure~\ref{fig3} we present the sample standard deviation of the estimator, which exhibits a polynomial decay. The solid black line corresponds to the theoretical standard deviation $\theta\sqrt{2/m}$ given by \eqref{eq:asynormThetaFixTime}, confirming the asymptotic normality result.

Next, we consider the estimator $\wt\theta_{n,m}$ given in \eqref{eq:sampleSpaceTime}, by assuming that the solution is observed on a space-time grid, over entire spacial domain, and time interval $[0,1]$. In Figure~\ref{fig4} we present the values of  $\wt\theta_{n,m}$, as function of $m$ (number of space discretization points) for $n=100$ and $n=500$ (number of time discretization points). Clearly, the rate of convergence of the estimators to the true parameter is significantly faster than using observations at just one time point.

Similar plots, and conclusions are performed for $\sigma$, assuming $\theta$ is known; see the right panels of Figure~\ref{fig1}-\ref{fig4}. Analogous results were obtained for sampling scheme (A), and for brevity we omit presenting the plots here.

Finally, we address the problem of estimating simultaneously $\theta$ and $\sigma$, by using the estimators \eqref{eq:Simultan}. The estimators are displayed in Figure~\ref{fig5}. Similar to the previous examples, we plot the estimates $\bar\theta_{n,m}, \bar\sigma^2_{n,m}$ as functions of number of observed points $m$ in space variable $x$, and for several values of the number of points $n$ in time variable. As $n$ and $m$ increases the estimates converge to the true values of the parameters.

\begin{figure}[h]
\centering
\includegraphics[width=.99\linewidth]{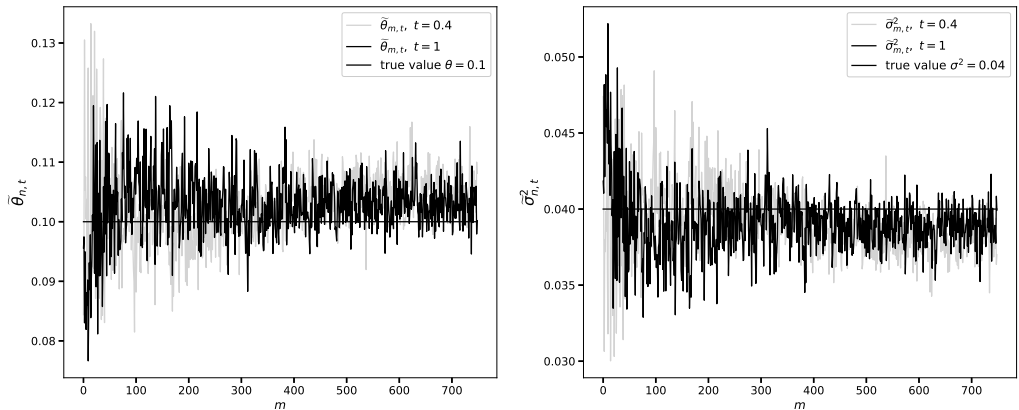}
\caption{{\footnotesize Sampling scheme (B). Sample path of $\wt\theta_{m,t}$ and $\wt\sigma_{m,t}^2$ for $t=0.4$ and $t=1$. }}
\label{fig1}
\end{figure}

\begin{figure}[h]
\centering
\includegraphics[width=.99\linewidth]{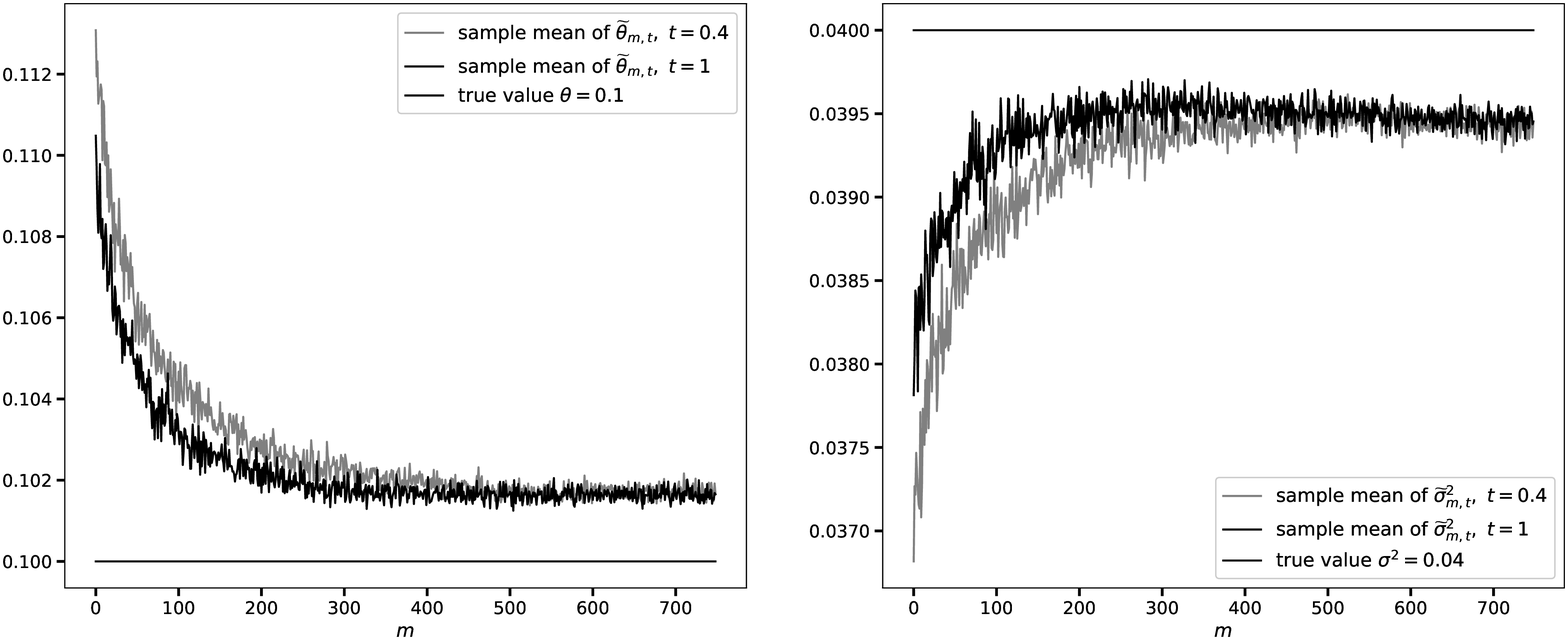}
\caption{{\footnotesize Sampling scheme (B). Sample mean of $\wt\theta_{m,t}$ and $\wt\sigma_{m,t}^2$ for $t=0.4$ and $t=1$. }}
\label{fig2}
\end{figure}

\begin{figure}[h]
\centering
\includegraphics[width=.99\linewidth]{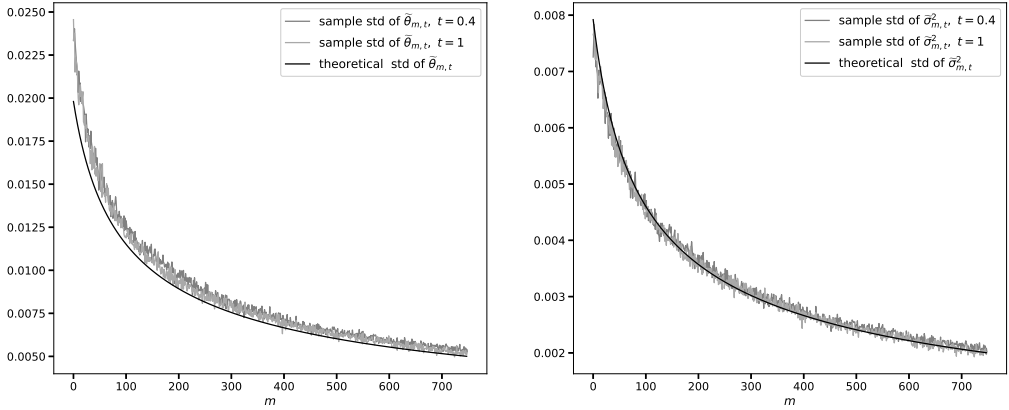}
\caption{{\footnotesize Sampling scheme (B). Sample standard deviation of $\wt\theta_{m,t}$ and $\wt\sigma_{m,t}^2$ for $t=0.4$ (grey) and $t=1$ (dark grey). Solid black lines are theoretical standard deviation from asymptotic normality. }}
\label{fig3}
\end{figure}

\begin{figure}[h]
\centering
\includegraphics[width=.99\linewidth]{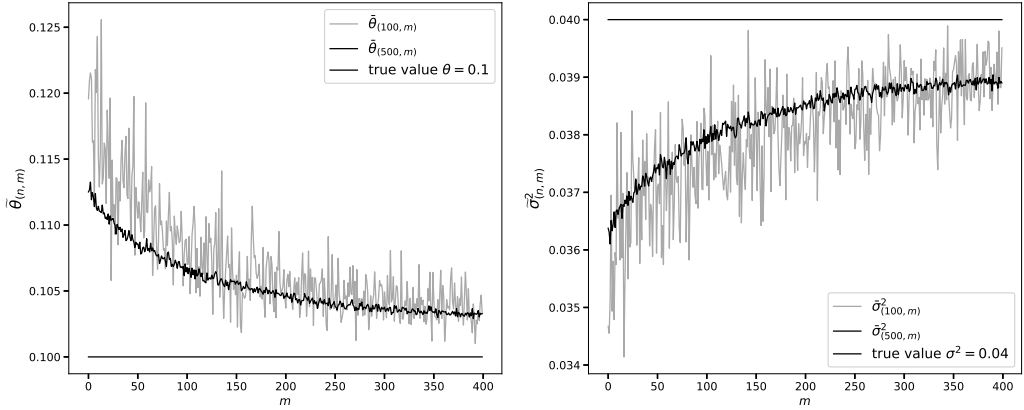}
\caption{{\footnotesize Sampling scheme (B). Sample path of $\wt\theta_{n,m}$ and $\wt\sigma_{n,m}^2$, for $t\in[0,1]$, and $n=100, \ n=500$. }}
\label{fig4}
\end{figure}

\begin{figure}[h]
\centering
\includegraphics[width=.99\linewidth]{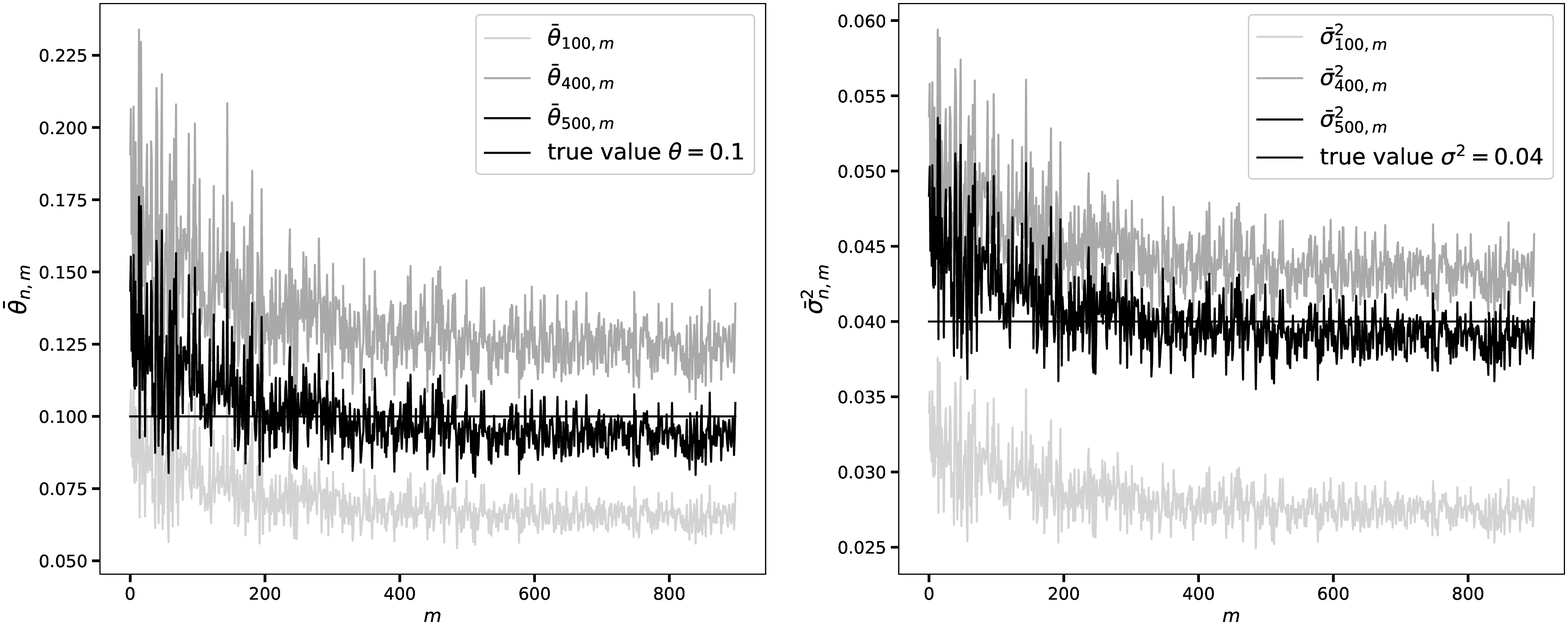}
\caption{{\footnotesize Joint estimation of $\theta$ and $\sigma$. Values of $\bar\theta_{n,m}$ (left panel) and $\bar\sigma^2_{n,m}$ (right panel) as function of $m$ for different values of $n$. Light grey corresponds to $n=100$, darker grey to $n=400$, and black to $n=500$, along the true value of the parameters (horizontal lines).}}
\label{fig5}
\end{figure}

\begin{appendix}

\section{Appendix}\label{appendix:auxiliaryResults}

\subsection*{Proof of Proposition~\ref{prop:QVarSmoothPert}}
First we prove \eqref{eq:SamePVarSmoothPert}. For a similar result  see also \cite[Corollary~2]{CorcueraNualartWoerner2006}. We only outline out proof here. All `$p$-variations' below are on the fixed interval $[a,b]$, and we will omit writing their dependence on $[a,b]$.
By Minkowski's inequality, we have that
\begin{align}
\mid \left(\mV_{n}^{p}(X)\right)^{1/p}-\left(\mV_{n}^{p}(Y)\right)^{1/p}\mid \leq \left(\mV_{n}^{p}(X+Y)\right)^{1/p}\leq \left(\mV_{n}^{p}(X)\right)^{1/p}+\left(\mV_{n}^{p}(Y)\right)^{1/p}.\label{eq:proof2.1varsumlesssumvar}
\end{align}
Since $Y$ has $C^{1}[a,b]$ sample paths, we have $\lim_{n\rightarrow\infty}\mV_{n}^{p}(Y)=0$. Hence, passing to the limit in \eqref{eq:proof2.1varsumlesssumvar}, the identity   \eqref{eq:SamePVarSmoothPert} follows.

As far as \eqref{eq:SamePVarSmoothPertbP}, note that in view of \eqref{eq:proof2.1varsumlesssumvar}, for any $\epsilon>0$,
\begin{align}
\Big\{ \big|  (\mV_{n}^{p}&(X+Y))^{1/p}   - (\mV_{\bP}^{p}(X))^{1/p} \big|\geq \epsilon \Big\} \\
&=\left\{\left(\mV_{n}^{p}(X+Y)\right)^{1/p}\geq \left(\mV_{\bP}^{p}(X)\right)^{1/p}+\epsilon  \right\}\cup \left\{\left(\mV_{n}^{p}(X+Y)\right)^{1/p}\leq \left(\mV_{\bP}^{p}(X)\right)^{1/p}-\epsilon  \right\}\\
&\subset\left\{\left(\mV_{n}^{p}(X)\right)^{1/p}+\left(\mV_{n}^{p}(Y)\right)^{1/p}\geq \left(\mV_{\bP}^{p}(X)\right)^{1/p}+\epsilon  \right\} \\
& \qquad \qquad \cup \left\{ \left| \left(\mV_{n}^{p}(X)\right)^{1/p}-\left(\mV_{n}^{p}(Y)\right)^{1/p} \right| \leq \left(\mV_{\bP}^{p}(X)\right)^{1/p}-\epsilon  \right\}\\
&\subset\left\{ \left| \left(\mV_{n}^{p}(X)\right)^{1/p}+\left(\mV_{n}^{p}(Y)\right)^{1/p}-\left(\mV_{\bP}^{p}(X)\right)^{1/p}\right|\geq \epsilon  \right\}\\
&\qquad \qquad \cup \left\{\left|  \left(\mV_{n}^{p}(X)\right)^{1/p}-\left(\mV_{n}^{p}(Y)\right)^{1/p}-\left(\mV_{\bP}^{p}(X)\right)^{1/p}\right|\geq \epsilon  \right\}\\
&=\left\{\left|\left(\mV_{n}^{p}(X)\right)^{1/p}-\left(\mV_{\bP}^{p}(X)\right)^{1/p}\right|\geq \epsilon/2  \right\}\cup \left\{\left(\mV_{n}^{p}(Y)\right)^{1/p}\geq \epsilon/2  \right\}. 
&
\label{eq:proof2.1setinclusion}
\end{align}
Due to the continuity of $x^{1/p}$, based on our initial assumptions, we have that
$\bP-\lim_{n\rightarrow\infty}\left(\mV_{n}^{p}(X)\right)^{1/p}=\left(\mV_{\bP}^{p}(X)\right)^{1/p}$, and $\bP-\lim_{n\rightarrow\infty}\left(\mV_{n}^{p}(Y)\right)^{1/p}=0$.
Thus, by \eqref{eq:proof2.1setinclusion}, we get at once that 
\begin{align}
\bP-\lim_{n\rightarrow\infty}\left(\mV_{n}^{p}(X+Y)\right)^{1/p}=\left(\mV_{\bP}^{p}(X)\right)^{1/p},
\end{align}
which consequently implies \eqref{eq:SamePVarSmoothPertbP}.

In view of Slutsky's Theorem, to prove \eqref{eq:CLTSmoothPert2}, it is enough to show that
\begin{align}
\lim_{n\rightarrow\infty}n^\alpha \left(  \mV^p_n(X+Y)   - \mV_n^p(X) \right) =0.
\end{align}
By \eqref{eq:proof2.1varsumlesssumvar} and by mean-value theorem, we have
\begin{align}
\mV_{n}^{p}(X+Y)&\leq \left(\left(\mV_{n}^{p}(X)\right)^{1/p}+\left(\mV_{n}^{p}(Y)\right)^{1/p}\right)^{p} \nonumber \\
&=\mV_{n}^{p}(X)+p\left(\left(\mV_{n}^{p}(X)\right)^{1/p}+\eta_{1,n}\left(\mV_{n}^{p}(Y)\right)^{1/p}\right)^{p-1}\left(\mV_{n}^{p}(Y)\right)^{1/p}, \label{eq:ap10}
\end{align}
for some $\eta_{1,n}\in [0,1]$. Since $Y$ has $C^{1}[a,b]$ sample paths, denoting
$M=\sup_{a\leq t\leq b}\mid Y'(t)\mid $, and again by mean-value theorem, we get
\begin{align}\label{eq:apB100}
\mV_{n}^{p}(Y)=\sum_{j=1}^{n} |Y(t_j)-Y(t_{j-1})|^p=\sum_{j=1}^{n} |(t_j-t_{j-1})Y'(\zeta_{j})|^p\leq n(M/n)^{p}.
\end{align}
Therefore, by \eqref{eq:ap10}, and since $\alpha+1/p<1$, we conclude that
\begin{align}
n^\alpha \left(  \mV^p_n(X+Y)   - \mV_n^p(X) \right)&\leq p\left(\left(\mV_{n}^{p}(X)\right)^{1/p}+\eta_{1}\left(\mV_{n}^{p}(Y)\right)^{1/p}\right)^{p-1}n^{\alpha+1/p-1}M \underset{n\to\infty}{\longrightarrow} 0.
\end{align}
Similarly, we have that
\begin{align}
n^\alpha \left(  \mV^p_n(X+Y)   - \mV_n^p(X) \right)&\ge -p\left(\left(\mV_{n}^{p}(X)\right)^{1/p}-\eta_{2}\left(\mV_{n}^{p}(Y)\right)^{1/p}\right)^{p-1}n^{\alpha+1/p-1}M\underset{n\to\infty}{\longrightarrow} 0,
\end{align}
and therefore, \eqref{eq:CLTSmoothPert2} is proved.

Now suppose that $Y$ has $C^2[a,b]$ sample paths, and assume that \eqref{eq:CLTSmoothPert} holds true for $p=2,\alpha=1/2$. To show that \eqref{eq:CLTSmoothPert2} also holds true, it is enough to prove that
\begin{align}\label{eq:apB2}
\lim_{n\rightarrow\infty}n^{1/2} \left(  \mV^2_n(X+Y)   - \mV_n^2(X) \right) =0.
\end{align}
Note that, $\mV^2_n(X+Y)   - \mV_n^2(X) =2\sum_{j= 1}^{n}\left(X(t_j)-X(t_{j-1})\right)\left(Y(t_j)-Y(t_{j-1})\right)+\mV_n^2(Y)$.
Using \eqref{eq:apB100}, we have $n^{1/2}\mV_n^2(Y)\leq n^{3/2}(M/n)^2\rightarrow 0$.

By mean value theorem,
\begin{align}
n^{1/2}\sum_{j=1}^{n}& \left(X(t_{j})-X(t_{j-1}) \right)\left(Y(t_{j})-Y(t_{j-1}) \right)
= n^{-1/2}(b-a)\sum_{i=1}^{n}\left(X(t_{j})-X(t_{j-1}) \right)\left(Y'(\zeta_{j})-Y'(t_{j-1})\right)\\
&  \qquad \qquad + n^{-1/2}(b-a)\sum_{i=1}^{n}\left(X(t_{j})-X(t_{j-1}) \right)Y'(t_{j-1}) =: K_{1}+K_{2}.
\end{align}
Applying Cauchy-Schwartz inequality, we get
\begin{align}
|K_{1}|
&\leq n^{-3/2}(b-a)^2\sum_{i=1}^{n}\Mid \left(X(t_{j})-X(t_{j-1}) \right)\max_{a\leq t\leq b}\mid Y''(t)\mid \Mid \\
&\leq n^{-1}(b-a)^2\max_{a\leq t\leq b}\mid Y''(t)\mid\sqrt{\mV_{n}^{2}(X)} \underset{n\to\infty}{\longrightarrow} 0 .
\end{align}
We rewrite $K_{2}$ as
\begin{align}
K_{2}=n^{-1/2}(b-a)\left(X(b)Y'(b)-X(a)Y'(a)-\sum_{j=1}^{n}X(t_{j})\left(Y'(t_{j})-Y'(t_{j-1})\right)\right).
\end{align}
Since,
$\lim_{n\rightarrow\infty}\sum_{j=1}^{n}X(t_{j})\left(Y'(t_{j})-Y'(t_{j-1})\right)=\int_{a}^{b}X(t)dY'(t)=\int_{a}^{b}X(t)Y''(t)dt,$
we have at once that
\begin{align}
\lim_{n\to\infty}K_{2}=\lim_{n\to\infty}n^{-1/2}(b-a)\left(X(b)Y'(b)-X(a)Y'(a)-\int_{a}^{b}X(t)Y''(t)dt\right)=0.\label{eq:proof2.1ine3}
\end{align}
Combining the above, \eqref{eq:apB2} is proved.

This concludes the proof.

\subsection*{Auxiliary technical results}
In this section we will provide some technical results used in the paper.
We will use the standard notations from \cite{Nualart2006} and \cite{NualartOrtiz2008}, and denote by $H(x;k)$ a polynomial with Hermite rank $k$, that is, $H$ can be expanded in the form
\begin{align}
H(x;k)=\sum_{j=k}^{\infty}c_{j}H_{j}(x),
\end{align}
where $c_{k}\ne 0$, and  $H_{j}$ is the $j$th Hermite polynomial (with leading coefficient $1$),
\begin{align}
H_j(x) = (-1)^j e^{\frac{x^2}{2}}\frac{d^j}{dx^j}(e^{-\frac{x^2}{2}}),\quad j\geq 1.
\end{align}
\begin{theorem}\label{thm:CLTforSelfSimilarGaussian}
Let $\{X_{t},t\ge 0 \}$ be a Gaussian process with the following properties
\begin{enumerate}[(i)]\addtolength{\itemsep}{-.3\baselineskip}
\item $X_{0}=0$, and $\bE X_{t}=0,\quad t\ge 0$.
\item $X_{t+s}-X_{t}\sim \cN(0,\sigma^{2}(s))$, where $\sigma(s)$ is a deterministic function of $s$.
\item There exists a constant $\gamma>0$ such that $\left(X_{\alpha t},t\ge 0 \right)\stackrel{\text{law}}{=}\alpha^{\gamma}\left(X_{t},t\ge 0\right)$, for any $\alpha>0.
$
\item For any $t\ge 0, \Delta t>0$, the sequence $X_{t+n\Delta t}-X_{t+(n-1)\Delta t}, \ n\in\bN$ is stationary. In particular, $Y_{n}=\frac{X_{n}-X_{n-1}}{\sigma(1)},\quad n\in\bN$,
is a zero mean and stationary Gaussian sequence with unit variance.

\item Let $r$ be the covariance function of $Y$,
$r(n)=\bE Y_{m}Y_{m+n}$, and assume that for some positive integer $k$,
$\sum_{n\ge 1}r^{k}(n)<\infty$.
\end{enumerate}
Then,
\begin{align}\label{eq:CLTforHermiteVariation}
\frac1{\sqrt{n}}\sum_{j=1}^{n}H\left(\frac{n^{\gamma}}{\sigma(1)}\left(X_{j/n}-X_{(j-1)/n}\right);k\right)\xrightarrow[n\rightarrow\infty]{\cD}\check{\sigma}\cN(0,1),
\end{align}
where \begin{align}
\check{\sigma}^{2}=\sum_{l=k}^{\infty}c_{l}^{2}l!\check{\sigma}_{l}^{2}, \qquad
\check{\sigma}_{l}^{2}=\lim_{n\rightarrow\infty}\frac1{n}\sum_{i=1}^{n}\sum_{j=1}^{n}r^{l}(|i-j|).
\end{align}
\end{theorem}
\begin{proof}
By \cite[Theorem~1]{BreuerMajor1983}, applied to the sequence $Y$, we immediately get
\begin{align}
\frac1{\sqrt{n}}\sum_{j=1}^{n}H(Y_{j};k)\xrightarrow[n\rightarrow\infty]{\cD}\check{\sigma}\cN(0,1),
\end{align}
where
\begin{align}
\check{\sigma}^{2}=\sum_{l=k}^{\infty}c_{l}^{2}l!\check{\sigma}_{l}^{2}, \qquad
\check{\sigma}_{l}^{2}=\lim_{n\rightarrow\infty}\frac1{n}\sum_{i=1}^{n}\sum_{j=1}^{n}r^{l}(|i-j|).
\end{align}
Since
\begin{align}
(X_{j/n}-X_{(j-1)/n},j=1,2,\ldots,n)\stackrel{\text{law}}{=}\frac1{n^{\gamma}}(X_{j}-X_{j-1},j=1,2,\ldots,n),
\end{align}
we conclude that \eqref{eq:CLTforHermiteVariation} holds.
\end{proof}

\begin{corollary}\label{corollary:sigmaForFBMwithH=1/4}
The following result is an immediate consequence of Theorem~\ref{thm:CLTforSelfSimilarGaussian}.
Let $B^{H}$ be a fractional Brownian motion with Hurst parameter $H=1/4$. Then,
\begin{align}
\sqrt{n}\left(\mV_{n}^{4}(B^{H};[a,b])-3(b-a) \right)\xrightarrow[n\rightarrow\infty]{\cD}(b-a)\check{\sigma}\cN(0,1),
\end{align}
where
\begin{equation}\label{eq:constCCheck}
\check{\sigma}^{2}=72\check{\sigma}_{2}^{2}+24\check{\sigma}_{4}^{2}, \qquad
\check{\sigma}_{l}^{2}=\lim_{n\rightarrow\infty}\frac1{n}\sum_{i=1}^{n}\sum_{j=1}^{n}r^{l}(|i-j|).
\end{equation}
\end{corollary}

For reader's convenience we also present here a result from \cite{NualartOrtiz2008}, used in the proof of Proposition~\ref{th:bdDecomTime}.
Let $H$ be a separable Hilbert space. For every $n\geq 1$,  the notation $H^{\otimes n}$ will stand for the $n$th tensor product of $H$, and $H^{\odot n}$ will denote the $n$th symmetric tensor product of $H$, endowed with the modified norm $\sqrt{n!}\norm{\cdot}_{H^{\otimes n}}$.
Suppose that $X=\{X(h),h\in H \}$ is an isonormal Gaussian process on $H$, on some fixed probability space, say  $(\Omega,\sF,\bP)$, and assume that $\sF$ is generated by $X$.

For every $n\geq 1$, let $\cH_{n}$ be the $n$th Wiener chaos of $X$, that is, the closed linear subspace of $L^{2}(\Omega,\cF,\bP)$ generated by the random variables $\{H_{n}(X(h)),h\in H,\norm{h}_{H}=1 \}$, where $H_{n}$ is the $n$th Hermite polynomial. We denote by $\cH_{0}$ the space of constant random variables. The mapping $I_{n}(h^{\otimes n})=H_{n}(X(h))$, for $n\geq 1$, provides a linear isometry between $H^{\odot n}$ and $\cH_{n}$. For $n=0$, we have that $\cH_{0}=\bR$, and take $I_{0}$ to be the identity map. It is well known that any square intergrable random variable $F\in L^{2}(\Omega,\sF,\bP)$ admits the following expansion
 \begin{align}
 F=\sum_{n=0}^{\infty}I_{n}(f_{n}),
 \end{align}
where $f_{0}=\bE F$, and the $f_{n}\in H^{\odot n}$ are uniquely determined by $F$.

Let $\{e_{k},k\geq 1 \}$ be a complete orthonormal system in $H$. Given $f\in H^{\odot n}$ and $g\in H^{\odot m}$, for $\ell = 0,\ldots,n \wedge m$, the contraction of $f$ and $g$ of order $\ell$ is the element of $H^{\otimes (n+m-2\ell)}$ defined by
\begin{align}
f\otimes_{\ell} g=\sum_{i_{1},\ldots,i_{\ell}}\langle f,e_{i_{1}}\otimes\cdots\otimes e_{i_{\ell}}\rangle_{H^{\otimes l}}\otimes\langle g,e_{i_{1}}\otimes\cdots\otimes e_{i_{\ell}}\rangle_{H^{\otimes l}}
\end{align}

\begin{theorem}[\cite{NualartOrtiz2008}]
\label{th:malliavin4thMomentThm}
For $d\geq 2$, fix $d$ natural numbers $1\leq n_{1}\leq \cdots\leq n_{d}$. Let $\{F_{k} \}_{k\in\bN}$ be a sequence of random vectors of the form
\begin{align}
F_{k}=(F_{k}^{1},\ldots,F_{k}^{d})=(I_{n_{1}}(f_{k}^{1}),\ldots,I_{n_{d}}(f_{k}^{d})),
\end{align}
where $f_{k}^{i}\in H^{\odot n_{i}}$ and $I_{n_{i}}$ is the Wiener integral of order $n_{i}$, such that, for every $1\leq i,j\leq d$,
\begin{align}\label{eq:NO081}
\lim_{k\rightarrow\infty}\bE\left[F_{k}^{i}F_{k}^{j} \right]=\delta_{ij}.
\end{align}
The following two\footnote{The original result \cite[Theorem~7]{NualartOrtiz2008} contains six equivalent conditions; we list only those two that we use in this paper.} statements are equivalent.
\begin{enumerate}[(N1)]
\item For all $1\leq i\leq d,1\leq \ell\leq n_{i}-1$, $\norm {f_{k}^{(i)}\otimes_{\ell} f_{k}^{(i)}} _{H^{2\otimes(n_{i}-\ell)}}^{2}\rightarrow 0$, as $k\rightarrow\infty$.
\item The sequence $\{F_{k} \}_{k\in\bN}$, as $k\to\infty$, converges in distribution to a $d$-dimensional standard Gaussian vector $\cN_{d}\left(0,I_{d} \right)$.
\end{enumerate}
\end{theorem}

We conclude this section with a result used to obtain the exact rates of convergence of some estimators from Section~\ref{sec:FixedX}.

\begin{lemma}\label{lemma:keylimit}
	For any $x\in (0,\pi)$ and $\theta>0$, the following holds true
	\begin{align}\label{eq:keylimit}
	\lim_{n\rightarrow\infty}\sqrt{n}\sum_{k\ge 1}\frac{\sin^{2}(kx)}{k^{2}}\left(1-e^{-\theta k^{2}/n}\right)=\frac{\sqrt{\pi\theta}}{2}.
	\end{align}
\end{lemma}
\begin{proof}
	Note that
	\begin{align}
	\sin^{2}(kx)=\frac12-\frac{\sin((2k+1)x)-\sin((2k-1)x)}{4\sin x},
	\end{align}
	and therefore,
	\begin{align}
	\sqrt{n}\sum_{k\ge 1}&\frac{\sin^{2}(kx)}{k^{2}}\left(1-e^{-\theta k^{2}/n}\right)\\
	&\sqrt{n}\sum_{k\ge 1}\frac1{2k^{2}}\left(1-e^{-\theta k^{2}/n}\right)-\sqrt{n}\sum_{k\ge 1}\frac{\sin((2k+1)x)-\sin((2k-1)x)}{4k^{2}\sin x}\left(1-e^{-\theta k^{2}/n}\right) \\
	& =: L_n^1 - L_n^2.
	\end{align}
	To prove \eqref{eq:keylimit}, we will show that $L^1_n\to \sqrt{\pi\theta}/2$, and $L^2_n\to0$.
	
	It is straightforward to check that for any $\varepsilon >0$, the function $(1-e^{- \epsilon x})/x$, $x>0$, is decreasing. It is also easy to show that
	\begin{align}
	\int_{0}^{\infty}\frac{1-e^{-z^{2}}}{z^{2}}dz=\sqrt{\pi}.
	\end{align}
	Using these, we obtain
	\begin{align}\label{eq:limit1upper}
	L_n^1 & =\sqrt{n}\sum_{k\ge 1}\int_{k-1}^{k}\frac1{2k^{2}}\left(1-e^{-\theta k^{2}/n}\right)\dif z \leq \sqrt{n}\sum_{k\ge 1}\int_{k-1}^{k}\frac1{2z^{2}}\left(1-e^{-\theta z^{2}/n}\right)\dif z \\
	&=\frac{\sqrt{n}}{2}\int_{0}^{\infty}\frac1{z^{2}}\left(1-e^{-\theta z^{2}/n}\right)\dif z
	=\frac{\sqrt{n}}{2}\int_{0}^{\infty}\frac1{y^{2}n/\theta}\left(1-e^{-y^{2}}\right)\dif y \sqrt{n/\theta}\\
	&=\frac{\sqrt{\theta}}{2}\int_{0}^{\infty}\frac1{y^{2}}\left(1-e^{-y^{2}}\right)\dif y
	=\frac{\sqrt{\pi\theta}}{2}.
	\end{align}
	On the other hand,
	\begin{align}\label{eq:limit1lower}
	L_n^1&=\sqrt{n}\sum_{k\ge 1}\int_{k}^{k+1}\frac1{2k^{2}}\left(1-e^{-\theta k^{2}/n}\right)\dif z
	\ge \sqrt{n}\sum_{k\ge 1}\int_{k}^{k+1}\frac1{2z^{2}}\left(1-e^{-\theta z^{2}/n}\right)\dif z\\
	&=\frac{\sqrt{n}}{2}\int_{1}^{\infty}\frac1{z^{2}}\left(1-e^{-\theta z^{2}/n}\right)\dif z
	=\frac{\sqrt{n}}{2}\int_{\sqrt{\theta/n}}^{\infty}\frac1{y^{2}n/\theta}\left(1-e^{-y^{2}}\right)\dif y \sqrt{n/\theta}\\
	&=\frac{\sqrt{\theta}}{2}\int_{\sqrt{\theta/n}}^{\infty}\frac1{y^{2}}\left(1-e^{-y^{2}}\right)\dif y \underset{n\to\infty}{\longrightarrow} \frac{\sqrt{\pi\theta}}{2}.
	\end{align}
	Combing \eqref{eq:limit1upper} and \eqref{eq:limit1lower}, we conclude that $L^1_n\to \sqrt{\pi\theta}/2$.
	
	Denote by
	\begin{align}
	f_{k}:=\frac{1-e^{-\theta k^{2}/n}}{k^{2}},\quad k\ge 1,
	\end{align}
	and as above, one can show that $\set{f_k,k\in\bN}$ is a decreasing sequence.
	By simple rearrangement of terms, we get
	\begin{align}
	L_2^n =\sqrt{n}\sum_{k\ge 2}\sin((2k-1)x)\left(f_{k-1}-f_{k}\right)-\sqrt{n}\sin xf_{1}.
	\end{align}
	Thus,
	\begin{align}
	|L_n^2|& \leq \sqrt{n}\sum_{k\ge 2}\Mid\sin((2k-1)x)\Mid\left(f_{k-1}-f_{k}\right)+\sqrt{n}\sin xf_{1}\\
	& \leq \sqrt{n}\sum_{k\ge 2}\left(f_{k-1}-f_{k}\right)+\sqrt{n}f_{1}
	\leq  2\sqrt{n}f_{1} = 2\sqrt{n}\left(1-e^{-\theta/n}\right)\\
	& \leq  2\sqrt{n}\frac{\theta}{n} = 2\frac{\theta}{\sqrt{n}} \underset{n\to\infty}{\longrightarrow} 0.
	\end{align}
	The proof is complete.
\end{proof}

\end{appendix}

\section*{Acknowledgments}
The authors would like to thank Prof. Robert C. Dalang and Prof. Sergey V. Lototsky for fruitful discussions that lead to some of the questions investigated in this manuscript. The authors are also grateful to the editors and the anonymous referee for their helpful comments and suggestions which
helped to improve the paper.

\bibliographystyle{alpha}
{\small
\def\cprime{$'$}

}

\end{document}